\documentclass[11pt,a4paper]{amsart}
\usepackage[utf8]{inputenc}

\usepackage[bb=ams,scr=euler]{mathalpha}

\usepackage{amsmath,amsthm,amssymb}
\usepackage{hyperref}

\usepackage{tikz}
\usetikzlibrary{cd,arrows.meta,calc,decorations.markings,patterns}

\usepackage{graphicx}

\newtheorem{theorem}{Theorem}
\newtheorem{lema}[theorem]{Lemma}
\newtheorem{prop}[theorem]{Proposition}
\newtheorem{cor}[theorem]{Corollary}
\newtheorem{obs}[theorem]{Remark}
\theoremstyle{definition}
\newtheorem{example}[theorem]{Example}
\usepackage[font=small,labelfont=bf]{caption}

\newcommand{\Ad}{\operatorname{Ad}}
\newcommand{\Z}{\mathbb{Z}}

\newcommand{\Tel}{\operatorname{Tel}}
\newcommand{\C}{\mathbb{C}}
\newcommand{\R}{\mathbb{R}}
\newcommand{\Id}{\operatorname{Id}}
\newcommand{\crit}{\operatorname{crit}}
\newcommand{\ind}{\operatorname{ind}}
\newcommand{\norm}[1]{\left\lVert#1\right\rVert}

\newcommand{\Hom}{\operatorname{Hom}}

\title[Equivariant and Reduced Floer homology]{Equivariant Floer homology is isomorphic to reduced Floer homology}
\author{Julio Sampietro Christ}
\email{julio.sampietro-christ@universite-paris-saclay.fr}
\address{Institut de mathématique d'Orsay, Université Paris-Saclay, Bâtiment 307, rue Michel Magat, Orsay, France}

\begin{document}
\raggedbottom
\maketitle

\begin{abstract}
  Given a symplectic manifold equipped with a Hamiltonian $G$-action and two $G$-invariant Lagrangians, we lift the construction of equivariant Lagrangian Floer homology of G.\@~Cazassus to the Novikov ring by constructing a ``quantum'' model in the vein of Biran and Cornea and Schm\"aschke. Using this refinement, we prove Cazassus' conjecture that, if the action on the zero-level of the moment map is free, then equivariant Floer homology agrees with the Floer homology of the Marsden-Weinstein reduction.
\end{abstract}

\section{Introduction}

Equivariant Lagrangian Floer homology is a homological invariant associated to a symplectic manifold $X$ with a Hamiltonian $G$-action and two $G$-invariant Lagrangians $L_{0},L_{1}$. Over the years, multiple versions of this invariant have appeared. The construction we are concerned with is due to \cite{Cazassus}, and is based on the quilted Lagrangian Floer theory of \cite{Wehrheim-Woodward2,Wehrheim-Woodward1}. Motivated by analogous results in classical equivariant homology, Cazassus conjectured that, when the group $G$ acts freely on the zero-level of the moment map $\mu$ of the $G$-action, and both Lagrangians $L_{0},L_{1}$ are contained in the zero-level, then the equivariant homology agrees with the regular Floer homology of the reductions of $L_{0}$ and $L_{1}$, namely, the Lagrangian submanifolds $\overline{L_{0}}$ and $\overline{L_{1}}$ of the symplectic reduction $X//G:=\mu^{-1}(0)/G$. This paper resolves the conjecture affirmatively by introducing a refined construction of equivariant Lagrangian Floer homology over the Novikov ring, inspired by the ``quantum Lagrangian homology'' model of \cite{Biran-Cornea,Biran-Cornea2}, and, more generally, the cascade homology for cleanly intersecting\footnote{See Section \ref{sec:floer-homology-of-lagrangians-in-clean-inntersection} for the definition of clean intersection.} Lagrangians developed by \cite{Schmaschke}. Our approach not only proves the conjecture but also sharpens the theory’s tools, offering enhanced control over the energy filtration (as is typically the case when working with ``quantum'' versions of Lagrangian Floer homology).

Lagrangian Floer homology, pioneered by Floer in \cite{Floer} studies Lagrangian submanifolds in symplectic topology. For transversally intersecting Lagrangians, it is the homology of a complex generated by the finitely many intersection points, with a differential counting rigid pseudoholomorphic strips. For non-transversally intersecting Lagrangians, quantum models like those in \cite{Oh4, Cornea-Lalonde, Biran-Cornea,Schmaschke} are the homologies of complexes obtained by deforming the Morse homology differential using so-called pearly trajectories combining holomorphic strips and Morse flow lines. In this paper, we will adopt Schmäschke's version \cite{Schmaschke} of quantum homology for a pair of cleanly intersecting Lagrangians.

Since the symplectic manifold $X$ is equipped with a Hamiltonian $G$-action (here $G$ is some compact Lie group) one would like an invariant that is sensitive to this additional structure, akin to ``classical'' equivariant homology. One may ask, for example, if a certain Lagrangian is equivariantly displaceable, i.e., if there exists a $G$-invariant function whose Hamiltonian isotopy displaces $L$ from itself. Another motivation is that the equivariant homology of Lagrangians in the zero level set gives a tool to probe potentially singular Lagrangians in the potentially singular quotient $X//G$. The case when the singularities are of orbifold type is relevant to the Atiyah-Floer conjecture \cite{Atiyah}.

With the goal of studying such problems, \emph{equivariant Lagrangian Floer homology} was developed. Several versions have appeared in the literature, for example \cite{Frauenfelder,KLZ,LLL,Cazassus}. In the last reference, Cazassus builds a finite-dimensional approximation of a ``symplectic Borel space,'' and takes a directed homotopy colimit of the finite-dimensional approximations using maps constructed with the pseudoholomorphic quilts of \cite{Wehrheim-Woodward2,Wehrheim-Woodward1}. His construction uses Hamiltonian perturbations at each step, so as to make the complexes generated by intersection points of transversally intersecting Lagrangians. A main contribution of this article is to provide a ``quantum'' model of his equivariant Lagrangian Floer theory, assuming the Lagrangians are in clean intersection. This allows the incorporation of coefficients over the Novikov ring and finer control over the energy filtration.

\subsection{Statement of the main result}
\label{sec:stat-main-result}

To motivate the main theorem, recall that in the equivariant homology of finite dimensional manifolds, if the action on the manifold is free, then the equivariant homology agrees with the homology of the quotient. Cazassus conjectured the same should hold for equivariant Floer homology: if the action of $G$ on the zero-level of the moment map is free, then the equivariant Floer homology is isomorphic to the regular Floer homology of the reduction. Proving this conjecture is the main application of our paper.

Suppose $X$ is a symplectic manifold equipped with a Hamiltonian $G$-action, and $L_0,L_1\subset X$ are two $G$-invariant Lagrangians contained in the zero-level of the moment map which intersect cleanly. We further assume that $\omega\cdot \pi_2(X,L_j)=0$ (said to be \emph{relatively exact}) and $X$ is either compact or convex at infinity.
\begin{theorem}\label{MainTheorem}
  If the action of $G$ on the zero-level of the moment map is free, then there is an isomorphism
  $$FH^{G}_{\ast}(L_0,L_1;\Lambda_0)\xrightarrow{\cong} FH_{\ast}(\overline{L_0},\overline{L_1};\Lambda_0).$$
  Where $\Lambda_0$ denotes the Novikov ring over $\Z_2$ and $\overline{L_j}=L_j\slash G$.
\end{theorem}
\begin{obs}
  While the theorem is stated for relatively exact Lagrangians, the proof carries over with minor modifications to the monotone case, where a polynomial ring can be used instead of $\Lambda_0$, and to the exact case where we can work over $\Z_2$ directly. If the Lagrangians are not in clean intersection, the same result holds with $\Lambda=\Lambda_0[T^{-1}]$-coefficients.
\end{obs}

\subsection{Strategy of proof}
The key idea behind the proof is the construction of the quantum model for equivariant Lagrangian Floer theory, together with what may be described as a spectral sequence argument enabled by the finer control of the energy filtration. In more detail, once the quantum model is constructed, we ensure that the differential preserves the energy filtration by a fixed positive quantity $\hbar$. Using pseudoholomorphic quilts and the homotopy-theoretic description of the telescope, a map to the Floer homology of the reduction is built. After verifying that this map also preserves the energy filtration of both Floer complexes by a fixed amount, it suffices to show that the map on the first page of the spectral sequence is an isomorphism. We show that the induced map on the first page is the classical isomorphism between the equivariant homology of a manifold and the homology of its quotient.

\subsection{Organization of the paper}

Section 2 is devoted to preliminaries, such as homological algebra over the Novikov ring, a review of Floer homology for Lagrangians in clean intersection and basics on pseudoholomorphic quilts. Section 3 develops the quantum model of equivariant Lagrangian Floer homology and section 4 constructs the isomorphism of the theorem.

\subsection{Acknowledgements}
This paper is part of the author's PhD thesis, and he is indebted to his supervisor, Claude Viterbo, for invaluable support. The author also wishes to thank Guillem Cazassus, Felix Schmäschke and Mohammed Abouzaid for useful conversations and Dylan Cant for suggesting valuable modifications. Finally he would like to thank Paul Biran, for pointing out an argument in his paper that simplified the zero-energy parts of the maps involved.

\tableofcontents

\section{Preliminaries}

\subsection{Algebra}
Throughout we work with $\Z_2$ as the underlying field.

\subsubsection{Novikov rings and gapped modules}
We define the \emph{universal Novikov ring} as the set of countable sums:
$$\Lambda_0:= \{\sum a_j T^{\lambda_j}\mid a_j\in \Z_2, \lambda_j\ge 0, \lim\limits_{j\to \infty} \lambda_j=+\infty\},$$
where $T$ is a formal variable. The \emph{Novikov field} $\Lambda$ is defined as the localization of $\Lambda_0$ with respect to the variable $T$:
$$\Lambda:= \Lambda_0 [T^{-1}]= \{\sum a_j T^{\lambda_j}\mid a_j\in \Z_2, \lambda_j\in \R, \lim\limits_{j\to \infty} \lambda_j=+\infty\}.$$

The Novikov ring is equipped with a natural decreasing filtration, which we refer to as the \emph{energy} (or \emph{area}) \emph{filtration}. Namely, for $\lambda \in \R$ we set:
$$F^{\lambda}\Lambda_0 =T^{\lambda}\Lambda_{0}= \lbrace \sum a_j T^{\lambda_j}\in \Lambda_0 \mid \lambda_j\geq \lambda\text{ for all }j\rbrace.$$
We denote by $\Lambda_0^{+}$ the ideal of all elements of strictly positive energy in $\Lambda_0$, so that $\Lambda_0\slash \Lambda_0^{+}\cong \Z_2$.

We now introduce \emph{free completed $\Lambda_0$-modules} and recall their basic theory. They already appeared in \cite{FOOO}, section $6$ under the name of ``dgfz.''

Starting with a $\Z_2$-vector space $\overline{C}$ one forms the free $\Lambda_0$-module $\overline{C}\otimes_{\Z_2}\Lambda_0$. It carries a natural filtration inherited from $\Lambda_0$, which we again denote by $F^{\lambda}$. The completion with respect to this filtration, written $C=\overline{C}\hat{\otimes}_{\Z_2}\Lambda_0$, is called a \emph{completed free $\Lambda_0$-module}.

In practice, the completion $C$ can be described as the set of formal sums $\sum v_j T^{\lambda_j}$ where $v_j\in \overline{C}$ and $\lim \lambda_j=+\infty$. Note that $C\slash{\Lambda_0^{+}C}\cong \overline{C}$ and there is a ``zero-energy lift`' $\overline{C}\to C$ by $v\mapsto v\otimes 1$.

If $C,D$ are completed free $\Lambda_0$-modules and $\varphi:C\to D$ is a $\Lambda_0$-linear map, then $\varphi$ induces a map $\overline{\varphi}:C\slash \Lambda_0^{+}C\to D\slash \Lambda_0^{+}D$ and by the aforementioned isomorphism a map $\overline{\varphi}:\overline{C}\to \overline{D}$. Such a map then lifts to a $\Lambda_0$-linear map $\hat{\varphi}:C\to D$ by the zero-energy lift. We say that $\varphi$ is $\hbar$-gapped (here $\hbar>0$ is a constant) if:
$$x\in F^{\lambda}C \implies (\varphi-\hat{\varphi})(x)\in F^{\lambda+\hbar}C.$$
If $\hbar$ is not relevant then we just say that $\varphi$ is \emph{gapped}. This is equivalent to the following statement: if $x$ is a zero-energy element in $C$, then
$$\varphi(x)= \overline{\varphi}(x)+ T^{\hbar}b(x), \hspace{2mm} b(x)\in D.$$

Of particular interest to us will be the case where $\partial:C \to C$ is a differential (i.e. $\partial^2=0$) and $\partial$ is gapped. When the differential is understood we will simply say that $C$ is a \emph{gapped differential completed free $\Lambda_0$-module}, or a ``gdcf.'' Note that the induced map $\overline{\partial}: \overline{C}\to \overline{C}$ is still a differential. A map $\varphi:C\to D$ between two gdcf's is a map that commutes with the differentials. The following proposition is due to \cite[Proposition 4.3.18]{FOOO}.
\begin{prop}\label{Prop1}
  If a map $\varphi:(C,\partial)\to (D, \partial)$ between two gdcf's is gapped, and the induced map:
  $$\overline{\varphi}: \overline{C}\to \overline{D}$$
  induces an isomorphism in homology, then $\varphi_{\ast}:H(C,\partial)\to H(D,\partial)$ is an isomorphism.
\end{prop}
\subsubsection{Telescopes}
Even if we work in the ungraded case, we formulate this section in terms of chain complexes, as is done in \cite{Cazassus}.

Suppose we have a sequence of chain complexes of $\Z_2$-modules $\{C_{\bullet}^{n}\}_{n\in \mathbb{N}}$ together with chain maps $f_{n}:C_{\bullet}^{n}\to C_{\bullet}^{n+1}$. We define the \emph{telescope} of such a sequence as
$$\overline{\Tel}(C_{\bullet}^{n}, f_n):= \bigoplus_{n\in \mathbb{N}} C_{\bullet}^{n}\oplus \bigoplus_{n\in \mathbb{N}} q\cdot C^{n}_{\bullet-1}$$
where $q$ is a formal variable of degree $1$. We set $f:\bigoplus_{n\in \mathbb{N}} C_{\bullet}^{n}\to \bigoplus_{n\in \mathbb{N}} C_{\bullet}^{n}$ to be the direct sum of all the $f_{n}$. We endow $\overline{\Tel}(C_{\bullet}^{n}, f_n)$ with a differential by the formula
$$\partial(x+q\cdot y):= \partial x+ f(y)-y+q\cdot \partial y.$$
Note that we have inclusions $i_n: C^{n}_{\bullet}\to \overline{\Tel}(C_{\bullet}^{n},f_n)$ given by $x\mapsto x=x+q\cdot 0$ and note that the diagrams
$$\begin{tikzcd}
  C^{n}_{\bullet}\arrow{rr}{f_n} \arrow{dr}{i_n}& & C^{n+1}_{\bullet}\arrow{dl}{i_{n+1}}\\
  & \overline{\Tel}(C_{\bullet}^{n},f_n) &
\end{tikzcd}$$
are commutative.
\begin{prop}
  With the differential described above, $\overline{\Tel}(C_{\bullet}^{n},f_n)$ is a \emph{directed homotopy colimit} of the diagram given by the $f_n$'s. In other words, it satisfies the following universal property: For any chain complex $D_{\bullet}$ together with maps $\varphi_n: C^n_{\bullet}\to D_{\bullet}$, such that the diagrams:
  $$\begin{tikzcd}
    C^{n}_{\bullet} \arrow{rr}{f_n}\arrow{dr}{\varphi_n} & & C^{n+1}_{\bullet}\arrow{dl}{\varphi_{n+1}}\\
    & D_{\bullet} &
  \end{tikzcd}$$
  are homotopy commutative, then there is a unique (up to homotopy) map $\Phi: \overline{\Tel}(C_{\bullet}^{n},f_n)\to D_{\bullet}$ such that the relevant diagrams are all homotopy commutative.
\end{prop}
\begin{proof}
  We only show existence and leave uniqueness to the reader's verification. Set $H_n$ to be a chain homotopy between $\varphi_n$ and $\varphi_{n+1}\circ f_n$ and define $H:\bigoplus_{n\in \mathbb{N}} C^{n}_{\bullet}\to D_{\bullet+1}$ to be their direct sum. Now construct:
  \begin{eqnarray*}
    \Phi: \overline{\Tel}(C_{\bullet}^{n},f_n)& \to & D_{\bullet}\\
    x+ q\cdot y & \mapsto & \varphi(x)+ q\cdot H(y)
  \end{eqnarray*}
  where $\varphi(x)=\bigoplus_{n\in \mathbb{N}} \varphi_n$. We check that it is a chain map:
  \begin{align*}
    \Phi\circ \partial (x+q\cdot y)&= \Phi(\partial x+ f(y)-y+q\cdot \partial y)\\
                                   &= \varphi(\partial x)+ \varphi(f(y))-\varphi(y)+H(\partial y)\\
                                   &= \partial \varphi(x)+\partial H(y)\\
                                   &= \partial \circ \Phi(x+q\cdot y).
  \end{align*}
  The commutativity relations are easily checked.
\end{proof}
By virtue of being a filtered homotopy colimit, it is known that
$$H_{\ast}(\overline{\Tel}(C_{\bullet}^{n},f_n))\cong \lim\limits_{n\to \infty}H_{\ast}(C^{n}_{\bullet})$$
where the term on the right is just the direct limit of modules.

Moving onto the realm of free completed $\Lambda_0$-modules, given a sequence $\{C_{\bullet}^{n}\}_{n\in \mathbb{N}}$ of such modules together with $\Lambda_0$-linear maps $f_{n}:C_{\bullet}^{n}\to C_{\bullet}^{n+1}$, we define the \emph{completed telescope} as the completion of the usual telescope
$$\Tel(C_{\bullet}^{n},f_n):= \widehat{\overline{\Tel}}(C_{\bullet}^{n},f_n)=\overline{\Tel}(\overline{C_{\bullet}^{n}},f_n)\hat{\otimes}_{\Z_2}\Lambda_0.$$
It is a free completed $\Lambda_0$-module and satisfies a similar universal property as before, except this time $D_{\bullet}$ has to be a completed free $\Lambda_0$-module.

Note that if all the $f_n$ are gapped by a common constant, so is their telescope, and if all the chain homotopies in the universal property are gapped by the same constant, then so is the induced map.

\subsection{Symplectic geometry}

\subsubsection{Hamiltonian actions}
Let $(X,\omega)$ be a symplectic manifold. We follow the convention that if $H:X\to \R$ is a smooth map, the \emph{Hamiltonian} vector field associated to $H$, denoted by $X_H$ satisfies
$$\omega(X_H, \_ )= dH(\_ ).$$

Fix a compact Lie group $G$ with Lie algebra $\mathfrak{g}= T_e G$. We denote by $\mathfrak{g}^{\ast}$ the linear dual $\Hom(\mathfrak{g},\R)$ and the canonical pairing by
\begin{eqnarray*}
  \langle \cdot, \cdot \rangle :\mathfrak{g}^{\ast}\otimes \mathfrak{g}& \to & \R\\
  \alpha\otimes \xi & \mapsto & \alpha(\xi).
\end{eqnarray*}
There is a canonical representation of $G$ on $\mathfrak{g}$, the \emph{adjoint representation}, given by the formula:
$$\Ad_g \xi := \frac{d}{dt}\big\rvert_{t=0}g\exp(t\xi)g^{-1}.$$
The dual action, referred to as the \emph{coadjoint action}, is defined by:$$\langle \Ad_{g}^{\ast}\alpha, \xi\rangle= \langle \alpha, \Ad_{g^{-1}} \xi\rangle.$$

If $G$ acts on the left on a smooth manifold $M$, we write $L_g: M\to M$ for the map $p\mapsto g\cdot p$. Given an element $\xi\in \mathfrak{g}$, the action induces a \emph{fundamental vector field}:
$$X_{\xi}(p):= \frac{d}{dt}\big\rvert_{t=0} \exp(t\xi)\cdot p.$$

If $G$ acts on $X$ by symplectomorphisms (i.e. $L_g^{\ast}\omega=\omega$) we say the action is \emph{Hamiltonian} if in addition there exists a smooth map $\mu:X\to \mathfrak{g}^{\ast}$, called the \emph{moment map}, such that:
\begin{itemize}
\item for all $\xi \in \mathfrak{g}$, the fundamental vector field $X_{\xi}$ is a Hamiltonian vector field for the map $p\mapsto \langle \mu(p),\xi\rangle$; in symbols:
  $$\omega(X_{\xi},\_ )= d\langle \mu,\xi\rangle.$$
\item $\mu$ is $G$-equivariant with respect to the coadjoint action.
\end{itemize}
Note that the subspace $\mu^{-1}(0)$ is $G$-invariant. If the action of $G$ on this subspace is free, then $\mu$ is automatically submersive at $0$ and $\mu^{-1}(0)$ is a smooth submanifold, and so is the quotient $\mu^{-1}(0)\slash G$. We denote by $i:\mu^{-1}(0)\to X$ the inclusion and by $\pi: \mu^{-1}(0)\to \mu^{-1}(0)\slash G$ the quotient map.
\begin{theorem}[Marsden-Weinstein reduction]
  If the action of $G$ on $\mu^{-1}(0)$ is free, then there exists a unique symplectic form $\omega_{red}$ on $\mu^{-1}(0)\slash G$ such that
  $$i^{\ast}\omega= \pi^{\ast}\omega_{red}.$$
\end{theorem}
We write $X\slash \slash G$ instead of $\mu^{-1}(0)\slash G$ (even if it is \emph{not} a quotient of the whole of $X$) and call it the \emph{reduction} of $X$ by $G$.

\begin{example}\label{ExCot}
  If $M$ is a closed smooth manifold with a left $G$-action, then the cotangent bundle $T^{\ast}M$ inherits a Hamiltonian action:
  $$g\cdot (q,p):= \left( g\cdot q, L_{g^{-1}}^{\ast}p\right)$$
  with moment map $\langle\mu(q,p),\xi\rangle= p(X_{\xi}(q))$. If the action on $M$ is free, so is the action on $T^{\ast}M$ and the reduction is identified to the cotangent bundle of the quotient:
  $$(T^{\ast}M)\slash \slash G\cong T^{\ast}(M\slash G).$$
\end{example}
\begin{obs}\label{RemarkDiag}
  If $G$ acts Hamiltonianly on $X$ and $Y$, then the diagonal action on $X\times Y$ is Hamiltonian with moment map
  $$\mu= \mu_{X}+\mu_{Y}.$$
\end{obs}
\begin{obs}[Marsden-Weinstein correspondence]
  The image of the map
  $$\iota\times \pi: \mu^{-1}(0)\rightrightarrows X\times \overline{X\slash \slash G}$$
  is a Lagrangian correspondence which we call the Marsden-Weinstein correspondence.
\end{obs}
There is a bijection between:
\begin{itemize}
\item $G$-invariant Lagrangians contained in $\mu^{-1}(0)$, and
\item Lagrangians in $X//G$,
\end{itemize}
where the bijection is given by projection $\pi$.

\subsubsection{Floer homology of Lagrangians in clean intersection}
\label{sec:floer-homology-of-lagrangians-in-clean-inntersection}

We recall here the basics of Floer homology of Lagrangians that intersect cleanly. Recall that this means that $L_0\cap L_1$ is a smooth manifold for each $p\in L_0\cap L_1$ we have the equality
$$T_p(L_0\cap L_1)= T_p L_0 \cap T_p L_1.$$
Intuitively, if Floer theory corresponds to Morse theory of the action functional, clean intersection Floer theory corresponds to Morse-Bott theory of the action functional. This was sketched in \cite[Appendix C]{Frauenfelder} and a more thorough exposition in the monotone case can be found in \cite{Schmaschke}. The work of \cite{Biran-Cornea} develops the same ideas in the case $L_0=L_1$.

At this stage, we state two important hypotheses: First, we assume that $L_0$ and $L_1$ are closed and relatively exact:
\begin{equation}\label{RE}\tag{RE}
  \omega\cdot \pi_2(X,L_j)=0, \ j=0,1.
\end{equation}
Second, if $X$ is non-compact, we assume that it is \emph{convex at infinity}. This means there exists a proper map $\rho:X\to [0,\infty)$ and an $\omega$-tame almost complex structure $J_{C}$ such that, outside of a compact set, $\rho$ is $J_{C}$-weakly plurisubharmonic.

Third, we make the following technical hypothesis
\begin{equation}\label{TH}\tag{TH}
	\begin{split}
	\text{For any map } u:[0,1]\times[0,1]\to X\text{ with }\\ u(s,0)\subset L_0, u(s,1)\subset L_1, u(t,0)=u(t,1)\\
	\text{we have that } \omega(u)>0\Rightarrow \operatorname{Mas}(u)\geq 2.
	\end{split}
\end{equation}
Here $\operatorname{Mas}(u)$ denotes the Maslov index associated to a loop of pairs of Lagrangians, see for example \cite[Appendix C]{Schmaschke}. This hypothesis is not strictly necessary and it is made to ease the exposition. We explain in Appendix \ref{Appendix} how to remove hypothesis \eqref{TH}.

Moving onto the description of the complex: first, we decompose the intersection $L_0\cap L_1$ into its connected components:
$$L_0\cap L_1\cong \bigsqcup C_j$$
and fix a Morse function $f$ and a Riemannian metric $g$ on $L_0\cap L_1$ such that together they form a Morse-Smale pair (i.e. the stable and unstable manifold of any pair of critical points intersect transversely). We write $\psi^{t}$ to be the time $t$ flow of $-\nabla_g f$, and denote the stable and unstable manifolds of a critical point $p$ as $W^{s}(p)$ and $W^{u}(p)$ respectively. Set:
\begin{equation*}
\text{$\widetilde{\mathcal{M}}(p,q)=W^{u}(p)\cap W^{s}(q)$ and $\mathcal{M}(p,q)=\widetilde{\mathcal{M}}(p,q)\slash \R$},
\end{equation*}
where $\R$ acts by reparametrization of the flow lines.

Introduce the strip $\Theta:= \{s+it\mid t\in [0,1]\}\subset \C$, with its standard complex structure. Write $\mathcal{J}_{\operatorname{adm}}(X)$ for space of \emph{admissible} compatible almost-complex structures on $X$: outside some compact set, $J$ agrees with the fixed convex-at-infinity almost complex structure $J_C$.

Given a path $J_t\in \mathcal{J}_{\operatorname{adm}}(X)$, a \emph{Floer strip} is a smooth map $u:\Theta\to X$ such that:
\begin{equation}\label{FS}\tag{FS}
  \begin{cases}
    \frac{\partial u}{\partial s}+ J_t \frac{\partial u}{\partial t}=0\hspace{2mm} (du^{0,1}=0),\\
    u(s,0)\subset L_0, u(s,1)\subset L_1,\\
    E(u):= \frac{1}{2}\int_{\Theta} |du|^2 <\infty.
  \end{cases}
\end{equation}
The quantity $E(u)$ is called the \emph{energy} of $u$. For solutions of \eqref{FS}, the energy is a purely topological quantity:
$$E(u)=\int_{\Theta} u^{\ast}\omega\geq 0.$$
Note that $E(u)=0$ if and only if $u$ is constant.
\begin{prop}
  Any solution of \eqref{FS} converges exponentially fast as $s\to \pm \infty$ to intersection points in $L_0\cap L_1$.
\end{prop}
\begin{proof}
  See \cite[Section 3]{Schmaschke}.
\end{proof}
Given a Floer strip $u$, we denote by $u(+\infty)$, respectively $u(-\infty)$, the corresponding intersection points.

A \emph{parametrized cascade} (with $m$ jumps) from $p$ to $q$ (here $p,q\in \crit(f)$) representing a homotopy class $A$ is a tuple $(u_1,\dots, u_m)$ such that:
\begin{enumerate}
\item each $u_k$ is a non-constant Floer strip,
\item for each $j\in \{1,\dots, m-1\}$ there exists $t_j\geq 0$ such that
  $$\psi^{t_j}(u_{j}(+\infty))=u_{j+1}(-\infty).$$
\item We have $u_1(-\infty)\in W^{u}(p)$ and $u_{m}(+\infty)\in W^{s}(q)$.
\item No $u_k$ is a reparametrization of another $u_j$, i.e., if $k\neq j$, there does not exist $s$ such that $u_k(s+\cdot )= u_j(\cdot )$. Such a tuple will be called \emph{distinct}.
\item The sum of all homotopy classes of the $u_i$ represents $A$.
\end{enumerate}
We denote by $\widetilde{\mathcal{M}}_{m}(p,q,A)$ the space of such cascades, and observe that $\R^{m}$ acts on it by $(r_1,\dots, r_m)\cdot (u_1,\dots, u_m)= (u_1(r_1+\cdot),\dots, u_{m}(r_m+\cdot))$. Let $\mathcal{M}_{m}(p,q,A)$ be the quotient by this action: this is the space of \emph{cascades with $m$ jumps} from $p$ to $q$ and homotopy class $A$. We extend the definition by setting $\mathcal{M}_{0}(p,q)=\mathcal{M}(p,q)$ (recall this is the space of Morse flow lines from $p$ to $q$). Finally, we consider the union of all such spaces:
$$\mathcal{M}(p,q,A):=\bigcup_{m\in \mathbb{N}}\mathcal{M}_{m}(p,q,A).$$
See Figure \ref{figure:cascade-3-jumps} for an illustration of an element of $\mathcal{M}(p,q,A)$.

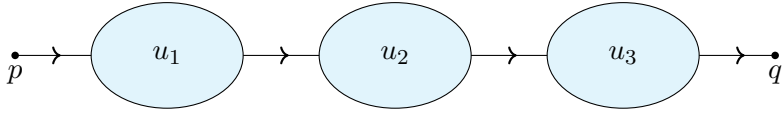
\begin{figure}[h]
  \centering
  \begin{tikzpicture}[scale=1]
    \path[decorate,decoration={
      markings,
      mark=between positions 0.06 and 1 step 0.3 with {\arrow{>[scale=1.3,line width=.8pt]};},
    }] (0,0) to (10,0);
    \draw (0,0) node[below] {$p$} node[fill, inner sep=1pt, circle] {}-- (10,0)node[fill, inner sep=1pt, circle] {}node[below]{$q$};
    \draw[fill=white!90!cyan] (2,0) circle (1 and 0.7) node{$u_{1}$} (5,0) circle (1 and 0.7) node{$u_{2}$} (8,0) circle (1 and 0.7) node{$u_{3}$};
  \end{tikzpicture}
  \caption{A cascade from $p$ to $q$ with three jumps.}
  \label{figure:cascade-3-jumps}
\end{figure}

\begin{theorem}\label{TransversalityI}
  For a generic path $J_t\in \mathcal{J}_{\operatorname{adm}}(X)$, the spaces $\mathcal{M}(p,q,A)$ are smooth manifolds with corners of the expected dimension:
  $$\dim \mathcal{M}_{u}(p,q,A)= \ind_f(x)-\ind_f(y)-1+\operatorname{Mas}(u),$$
  whenever this number is $\leq 1$. Here $\operatorname{Mas}$ denotes the Maslov-Viterbo index of $u$; see \cite[Appendix C]{Schmaschke}.

  Furthermore, when this dimension is zero the resulting manifold is compact, hence a finite set, and when this dimension is $1$, the manifold $\mathcal{M}_{[1]}(p,q,A)$ admits a compactification whose boundary is given by
  $$\partial \overline{\mathcal{M}}_{[1]}(p,q,A)\cong \bigsqcup_{\stackrel{z}{B+C=A}}\mathcal{M}_{[0]}(p,z,B)\times \mathcal{M}_{[0]}(z,q,C).$$
  (Here the brackets denote the component of the specified dimension).
\end{theorem}
\begin{proof}
  We only sketch the proof, this carried out in great detail in \cite[Section 6]{Schmaschke}.

  First we treat the transversality statement. Given a tuple of connected components $C:=C_{1},\dots, C_{2m}$ of $L_0\cap L_1$, we form a Banach space of tuples $(u_1,\dots, u_{m})$ of maps $u\in W^{1,p}_{\delta}(\Theta,X)$ with appropriate boundary conditions and such that $u_{k}(-\infty)\in C_{2k-1}, u_{k}(+\infty)\in C_{2k}$. Here $\delta$ denotes a positive constant such that $u_k$ converges $\delta$-exponentially fast to tangent vectors at $\pm\infty$. This is required to ensure that the Cauchy-Riemann operator is Fredholm, at least if $\delta$ is smaller than a fixed constant depending on $J$. We then we construct the Cauchy-Riemann operator:
  $$(u_1,\dots, u_m)\mapsto (\overline{\partial}_J u_1,\dots, \overline{\partial}_J u_m )\in L^{p}_{\delta}(u_1^{\ast}TX\otimes \Omega^{0,1}_{\Theta})\oplus\dots \oplus L^{p}_{\delta}(u_m^{\ast}TX \otimes \Omega^{0,1}_\Theta).$$

  Using a universal moduli space and the Sard-Smale theorem, one shows that restricted to the space of \emph{distinct tuples}, the subspace of paths of almost complex structures $J_{t}\in \mathcal{J}_{\operatorname{adm}}(X)$ such that:
  \begin{itemize}
  \item the linearization of the Cauchy-Riemann operator is surjective,
  \item the evaluation maps are submersive,
  \end{itemize}
  is of the second Baire category: a countable intersection of open dense sets. Since $\mathcal{J}_{\operatorname{adm}}$ is a Fréchet space this ensures this set is dense. We refer to the elements of this set as \emph{regular almost complex structures}.

  This ensures that for generic $J_{t}$, the space of distinct $J_t$-holomorphic tuples is a smooth manifold and the evaluation maps are submersions, which we denote $\mathcal{M}'_m(C)$. Then $\widetilde{\mathcal{M}}_{m}(p,q)$ may be described as the fiber product:
  $$\begin{tikzcd}
    \widetilde{\mathcal{M}}_{m}(p,q) \arrow{r} \arrow{d} & \mathcal{M}'_{m}(C)\arrow{d}{\operatorname{ev}}\\
    W^{u}(p)\times (L_0\cap L_1)^{m-1}\times [0,\infty)^{m-1}\times W^{s}(q) \arrow{r}{\Psi} & (L_0\cap L_1)^{2m}
  \end{tikzcd}$$
  where $\Psi:W^{u}(p)\times (L_0\cap L_1)^{m-1}\times[0,\infty)^{m-1}\times W^{s}(q)\to(L_0\cap L_1)^{2m}$ is given by:
  \begin{equation*}
    (x,p, t, y)\mapsto (x,p_1, \psi^{t_1}(p_1), p_2, \psi^{t_2}(p_2), \dots, \psi^{t_{m-1}}(p_{m-1}),y).
  \end{equation*}
  This describes $\widetilde{M}_{m}(p,q)$ as a smooth manifold with corners: the corners are given by the number of zeros in the tuple $(t_1,\dots, t_{m-1}$). Note that the dimension of a manifold with corners is defined as the dimension of the \emph{top stratum}, hence $\mathcal{M}_{[0]}(p,q,A)$ has no corners, and $\mathcal{M}_{[1]}(p,q,A)$ has at most one.
  
  Next we treat the compactness statement. The usual arguments of Floer-Gromov compactness apply, namely given a sequence of cascades with $m$ jumps in a fixed homotopy class (hence with uniform energy bounds) degenerates to one of the following 4 cases:
  \begin{enumerate}
  \item Some of the $u_k$ develop interior bubbles or boundary disc bubbles,
  \item some of the $u_{k}$ break into two strips,
  \item some gradient lines connecting consecutive strips shrinks to a point: this is a corner of the manifold,
  \item a gradient line breaks.
  \end{enumerate}
  We include pictures of the four cases:

\begin{center}
	\includegraphics[scale=0.5]{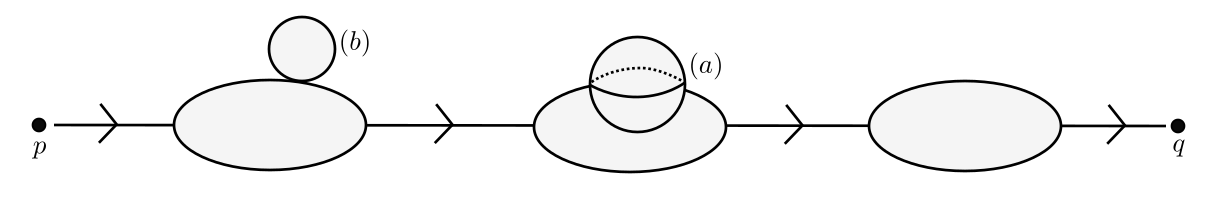}\\ \captionof{figure}{Case 1: A sequence of cascades may develop an interior bubble $(a)$ or a disc bubble $(b)$.}
\end{center}
\begin{center}
	\includegraphics[scale=0.5]{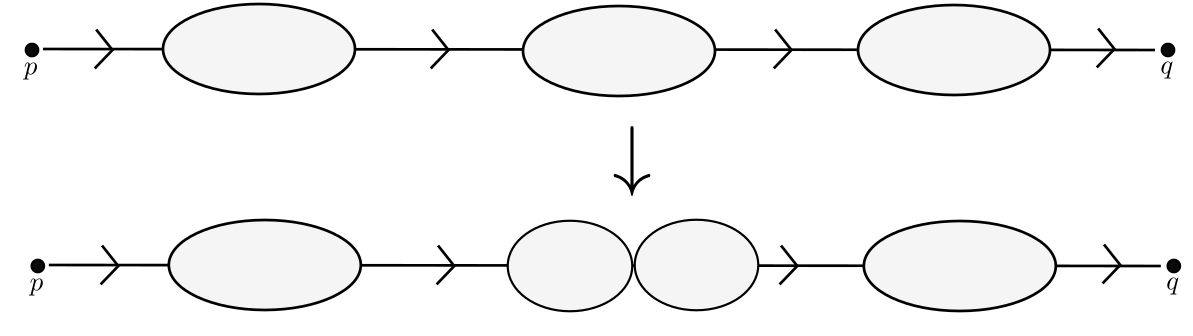}\\
	\captionof{figure}{Case 2: In a sequence of cascades, one or more $u_k$ may break.}
\end{center}
\begin{center}
	\includegraphics[scale=0.5]{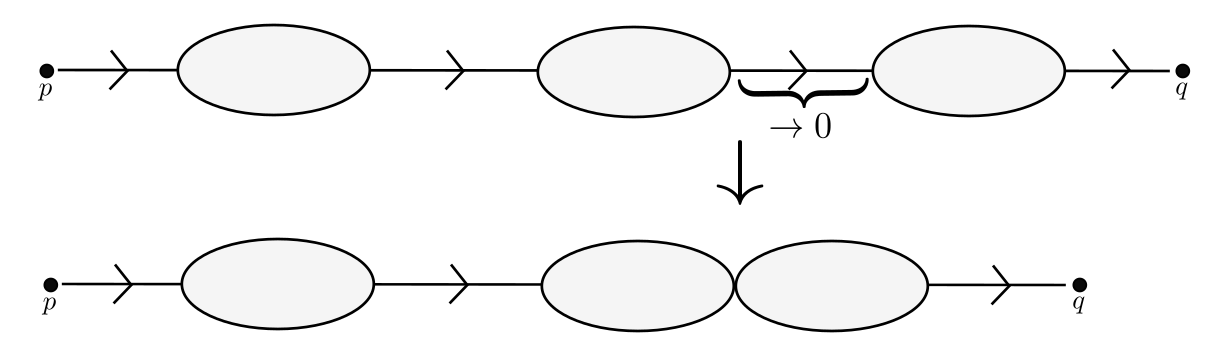}\\
	\captionof{figure}{Case 3: In a sequence of cascades, a gradient line may shrink to a point. }
\end{center}
\begin{center}
	\includegraphics[scale=0.5]{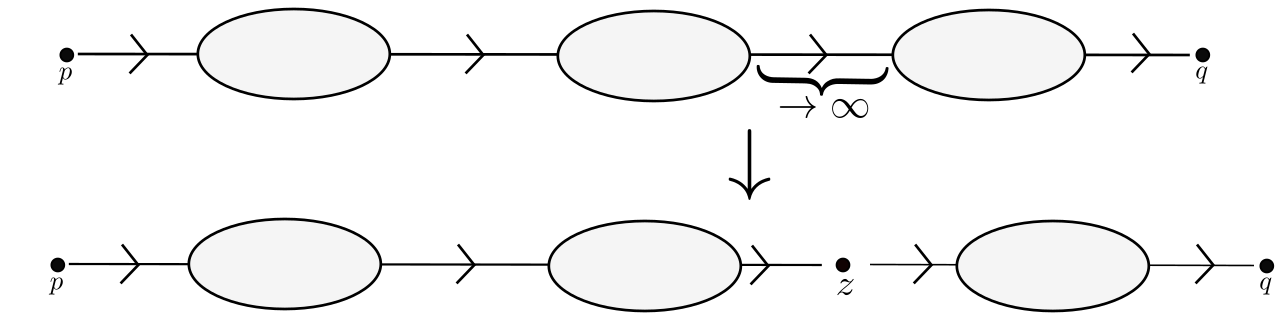}\\
	\captionof{figure}{Case 4: One of the gradient lines could break.}
\end{center}

  In our setting, case 1 can not occur since $\omega\cdot\pi_2(X,L_j)=0$.

  Looking at the zero-dimensional component: cases 2, 3 and 4 would be of strictly less dimension, so they do not happen for transversality reasons.

  In the one-dimensional component we can have at most one occurrence of each of the cases 2, 3 and 4. However: notice that cases 2 and 3 ``come in pairs:'' indeed if a disc breaks, then it can also be seen as a gradient line shrinking degeneration of the space of cascades with $m+1$ jumps. We can glue the compactifications of all the spaces $\mathcal{M}_{m}(p,q,A)$ along these identifications: they are no longer boundary points. Finally, the only degeneration that can appear in the boundary is case $4$, which proves the theorem.
\end{proof}
\begin{obs}
	We have implicitely used hypothesis \eqref{TH} by allowing ourselves to restrict our attention to distinct strips (i.e. no strip is a reparametrization of another), where we can achieve transversality. We could a priori have a sequence of cascades that converges to a cascade where two discs are reparametrization of one another. However if this was to happen then the following the starting point of the first one all the way to the endpoint of the second would produce a loop with boundaries in $L_0$ and $L_1$. By the \eqref{TH}, this loop has Maslov index greater or equal to 2. By removing it we get a cascade where all the discs are transversely cut-out since they are distinct: but the virtual dimension which coincides with the actual dimension by transversality has been reduced by $2$ and thus is $-1$, so empty.
\end{obs}
We define the \emph{Floer chain complex} as:
$$FC_{\ast}(L_0,L_1):= \operatorname{Span}_{\Z_2}\{p\mid p\in \crit f\}\otimes \Lambda_0.$$
Since the set $\{p\mid p\in \crit(f)\}$ is finite, there is no need to complete the tensor product (it is already complete).

For a generic (in the sense of the previous theorem) path $J_t\in \mathcal{J}_{\operatorname{adm}}(X)$ we define the \emph{Floer differential} as the $\Lambda_0$-linear extension of the map:
\begin{eqnarray*}
  \partial: \operatorname{Span}_{\Z_2}\{p\mid p\in \crit f\}&\to & FC_{\ast}(L_0,L_1)\\
  p & \mapsto & \sum_{q} \#_2 \mathcal{M}_{[0]}(p,q,A) T^{\omega(A)} \cdot q.
\end{eqnarray*}
By Gromov compactness, this sum is actually an element in $FC_{\ast}(L_0,L_1)$ (i.e. the energies diverge). The compactification of the 1-dimensional component serves to show that $\partial^2=0$.

Define the \emph{Floer homology} of $L_0$ and $L_1$ over $\Lambda_0$ as the homology of this differential module. We will denote it by $FH_{\ast}(L_0,L_1;\Lambda_0)$. If in the definition we had used $\Lambda$ instead of $\Lambda_0$ we would get the \emph{Floer homology} of $L_0$ and $L_1$ over $\Lambda$:
$$FH_{\ast}(L_0,L_1;\Lambda)\cong FH_{\ast}(L_0,L_1;\Lambda_0)\otimes_{\Lambda_0} \Lambda.$$
This isomorphism comes from the exactness of localization from $\Lambda_0$ to $\Lambda$.

The Floer homology groups satisfy the following (by now standard) properties:
\begin{itemize}
\item Over $\Lambda_0$ (and hence over $\Lambda$) they are independent, up to natural isomorphism, of the choice of auxiliary data: the generic $J_t$, the Morse function and the Riemannian metric.
\item Over $\Lambda$ (warning: in general, not over $\Lambda_0$) they are independent of the Hamiltonian isotopy class of $L_1$: if $\varphi^{1}(L_1)\cap L_0$ is clean, then there is a natural isomorphism
  $$FH_{\ast}(L_0,L_1;\Lambda)\cong FH_{\ast}(L_0,\varphi^{1}(L_1);\Lambda).$$
\item $FH_{\ast}(L,L,\Lambda_0)\cong MH_{\ast}(L)\otimes_{\Z_2} \Lambda_0$, where $MH_{\ast}$ denotes the Morse homology with $\Z_2$-coefficients.
\end{itemize}
We also have the following:
\begin{prop}\label{Gap1}
  Given a regular $J_t$, there exists a constant $\hbar>0$ such that any non-constant solution $u$ to \eqref{FS} has energy (hence area) $E(u)>\hbar$. In particular, the Floer complex is gapped.
\end{prop}
\begin{proof}
  Assume this was not the case: we would then have a sequence $u_{\nu}$ of $J_t$-holomorphic strips such that $E(u_{\nu})>0$ and $\lim\limits_{\nu\to \infty}E(u_{\nu})=0$. By Gromov compactness, by passing to a subsequence we may assume $u_{\nu}$ converges to some $u$. But then $E(u)=0$, so $u$ must be constant so for some $\nu_0$ large enough, $u_{\nu}$ would represent the zero homotopy class for $\nu\geq \nu_0$: this would imply $E(u_{\nu})=0$ for all $\nu\geq \nu_0$, a contradiction to our assumption that they have positive energy.
\end{proof}
This result also appears in \cite[Proposition 4.9]{Schmaschke} using Poźniak neighbourhoods.

The zero-energy part of the differential is just the count of gradient flow lines, in particular:
$$H_{\ast}(\overline{FC}(L_0,L_1;\Lambda_0))\cong MH_{\ast}(L_0\cap L_1)$$
where $MH_{\ast}$ is the Morse homology and the bar denotes the zero-energy complex.

A similar argument shows the $\hbar$ of the previous proposition is ``lower semi-continuous,'' in the following sense:
\begin{prop}\label{Gap2}
  Given a fixed $J_t$ with $\hbar>0$ as in the previous lemma and given $\varepsilon<\hbar$, there exists an open neighbourhood $U$ of $J_t$ in the space of smooth paths in $\mathcal{J}_{\operatorname{adm}}(X)$ such that for all $J_t'\in U$, any nonconstant $u$ has energy greater than $\hbar-\varepsilon$.
\end{prop}
\begin{proof}
  Apply Gromov compactness in a similar fashion to the previous proposition.
\end{proof}
\begin{obs}
  If bubbling was to happen (for example in the monotone case), then the same results hold by adjusting $\hbar$ to be the minimal area of all holomorphic strips and spheres.
\end{obs}
\subsubsection{Pseudoholomorphic quilts}
We review here the essential material concerning pseudoholomorphic quilts. This is by no means a complete exposition and we refer to the landmark article of Wehrheim and Woodward, \cite{Wehrheim-Woodward1}.

A \emph{surface with strip-like ends} consists of a compact Riemann surface with boundary, $\overline{S}$, together with three finite and disjoint subsets of points in the boundary:
$$\mathcal{E}=\mathcal{E}_{+}\sqcup \mathcal{E}_{-}\sqcup \mathcal{E}_{\operatorname{free}} \subset \partial \overline{S},$$
and a complex structure $j_{S}$ on $S:=\overline{S}\setminus \mathcal{E}$ such that for each point $z\in \mathcal{E}_{+}$ (resp. $\mathcal{E}_{-})$ there is an embedding $\epsilon_{z}: \R^{+}\times[0,\delta_{z}]\to S$ (resp. $\R^{-}\times [0,\delta_z]\to S$) such that:
\begin{itemize}
\item $\epsilon_{z}(\R^{\pm}\times\{0, \delta_{z}\})\subset \partial S$,
\item $\lim\limits_{s\to \pm \infty}\varepsilon_z(s,t)=z$,
\item $\varepsilon^{\ast}_z j_S=j_0$ where $j_0$ is the standard complex structure on the half-strip $\R^{\pm}\times [0,\delta_{z}]$.
\end{itemize}
Finally, the elements of $\mathcal{E}_{\operatorname{free}}$ are also equipped with strip-like ends, either positive or negative. Note that the distinction between a free end an incoming/outgoing end is merely that of labeling: in practice we won't care what happens at these ends.

We refer to elements of $\mathcal{E}_{+}$ as outgoing ends, those of $\mathcal{E}_{-}$ as incoming ends and those of $\mathcal{E}_{\operatorname{free}}$ as free ends.

A \emph{quilted surface with strip-like ends}, $\underline{S}$, consists of a surface with strip-like ends $S$, together with a collection of connected, pairwise disjoint subsets $\mathcal{S}$, called \emph{seams}, such that for every $s\in \mathcal{S}$ we have
\begin{itemize}
\item $s$ is a real-analytic subset of $S$,
\item on each strip-like end $\varepsilon_{z}:\R^{\pm}\times [0,\delta_z]\to S$ we have $\varepsilon_z^{-1}(s)$ is either empty or consists of one, or two disjoint lines of the form $\R^{\pm}\times\{t\}$ for $0<t<\delta_{z}$.
\end{itemize}
The closure of the connected components of $\operatorname{int}(S)\setminus\mathcal{S}$ are called \emph{patches} (written as $\mathcal{P}$), while connected components of $\partial S$ are called \emph{true boundaries}.

A \emph{decoration} of a quilted surface consists of the following data:
\begin{itemize}
\item for each patch $P\in \mathcal{P}$, a symplectic manifold $M_{P}$,
\item for each seam $s\in \mathcal{S}$ adjacent to two patches $P$ and $Q$, a \emph{Lagrangian correspondence} $L_{s}\subset M_{P}\times\overline{M_Q}$,
\item for each true boundary component $D$, a Lagrangian $L_D\subset M_P$ where $P$ is the unique patch adjacent to it.
\end{itemize}
See, e.g., Figure \ref{fig:dec-quilt-surf} for an illustration.

To look at pseudoholomorphic boundary value problems, we equip a decorated quilted surface $\underline{S}$ with the following data for each patch $P\in \mathcal{P}$:
\begin{itemize}
\item a smooth map $J_P:P\to \mathcal{J}_{\operatorname{adm}}(M_P)$ where $\mathcal{J}_{\operatorname{adm}}(M_P)$ is some class of admissible almost complex structures. We denote the tuple $(J_P)_{P\in \mathcal{P}}$ simply as $\underline{J}$, and by $\mathcal{J}(\underline{S})$ the set of all $\underline{J}$.
\item a smooth function-valued $1$-form: $$K_P\in \Omega^{1}(P, C^{\infty}(M_P))$$ such that $K_{P}\rvert_{\partial P}=0$, and such that, on each strip-like end $z$, $K_P$ is identified with $H_{z} dt$ for some smooth function $H_z$.
\end{itemize}
Denote the set of such tuples $(K_P)_{P\in \mathcal{P}}$ by $\underline{K}$ and by $\operatorname{Ham}(\underline{S}, (H_{z})_{z\in \mathcal{E}})$ the set of all tuples that on the strip-like ends agree with the given $H_z$. Note that to each $\underline{K}$ there is a tuple of Hamiltonian vector field-valued one-forms $X_{\underline{K}}= (X_{K_P}\in \Omega^1(P, \operatorname{Vect}(M_P)))_{P\in \mathcal{P}}$.

Given this data, we consider the moduli space of maps
$$\mathcal{Q}(\underline{S}):=\{\underline{u}=(u_P:P\to M_P)_{P\in \mathcal{P}}\mid \text{(1), (2), (3), (4)}\}$$
where:
\begin{enumerate}
\item $u_P$ is $(J_P,K_P)$-pseudoholomorphic: $$\overline{\partial}_{J_P,K_P}u_P= (du_P-X_{K_P})^{0,1}=0,$$
\item $u_{P}$ maps each true boundary components to the corresponding Lagrangian: $u_P(D)\subset L_D$, 
\item the maps $u_{P}$ are compatible at the seams: for a seam $s$ and two adjacent patches $P$ and $Q$ we have $(u_P,u_Q)\rvert_{s}\subset L_s$.
\item each $u_{P}$ has finite energy:
  $$E(\underline{u})=\sum_{P\in \mathcal{P}}\int_{P}\frac{1}{2}\left|du_P-X_{K_P}(u)\right|^2<\infty.$$
\end{enumerate}
Depending on the Lagrangians and if they are in transverse or clean intersection, the finite energy assumption implies any such map converges exponentially fast to $H_z$-perturbed intersection points on the strip like ends.

The standard process associates a Banach bundle over the Banach manifold of all tuples of maps of class $W^{1,p}$, for some $p>2$, whose fiber over $\underline{u}$ is the space of tuples $(L^{p}(u_P^{\ast}TM_P\otimes \Omega^{0,1}_{M_P}))_{P\in \mathcal{P}}$. The operator $\bigoplus\overline{\partial}_{J_P,K_P}u_P$ defines a section, and for a fixed $\underline{J}$, there exists a comeagre subset of Hamiltonian perturbations such that this section is transverse to the zero section and hence $\mathcal{Q}(\underline{S})$ will inherit the structure of a smooth manifold.

Finally, there is a notion of Gromov compactness for pseudoholomorphic quilts (see for example \cite{Lekili-Lipyanskiy}, discussion after lemma 6 for an explanation):
\begin{theorem}[Gromov Compactness for quilts]
  Given a sequence $\underline{u}_{\nu}$ of pseudoholomorphic quilts with uniformly bounded energy and converging perturbation data, there exists a subsequence that converges to a pseudoholomorphic quilt in the $\mathcal{C}^{\infty}_{\operatorname{loc}}$-topology, except at a finite number of points where the following can happen:
  \begin{itemize}
  \item At interior points a holomorphic bubble in the corresponding patch is formed.
  \item At true boundary points a disc with Lagrangian boundary condition is formed.
  \item At seam points, a disc with boundary in the seam condition is formed.
  \end{itemize}
\end{theorem}

\subsection{Equivariant Morse homology}

In this section we recall the Morse homology version of the Borel construction.

\subsubsection{Classical equivariant homology}
Any compact Lie group (or more generally any topological group) $G$ admits a \emph{universal bundle}. This is a principal $G$-bundle $EG\to BG$ such that for any space $M$ there is a natural bijection between:
\begin{enumerate}
\item principal $G$-bundles over $M$, up to isomorphism,
\item smooth maps $f:M\to BG$, up to homotopy.
\end{enumerate}
The space $EG$ can be characterized up to $G$ homotopy equivalence as the only contractible space with a left free $G$-action.\footnote{Here we follow the convention that principal bundles have a left action.}

Given a space $M$ with a left $G$-action, the so-called \emph{Borel space} is:
$$M\times_G EG:= \left(M\times EG\right)\slash G_{\operatorname{diag}}.$$
The (Borel) \emph{equivariant (co)homology groups} of $M$ are defined as:
$$H^{G}_{\ast}(M):= H_{\ast}(M\times_G EG),\hspace{2mm} H_G^{\ast}(M):= H^{\ast}(M\times_G EG).$$

The projection map $M\times EG\to EG$ is $G$-equivariant (with respect to the diagonal action) and hence descends to a map $p_2:M\times_G EG\to BG$. This gives $M\times_G EG$ the structure of a locally trivial fiber bundle over $BG$ with structure group $G$ and fiber $M$:
$$\begin{tikzcd}
  M\arrow[hook,r] & M\times_G EG\arrow{d}{p_2}\\
  & BG.
\end{tikzcd}$$
Similarly, the projection onto the first factor descends to a map of the form $p_1:M\times_G EG\to M\slash G$. Crucially, if the action of $G$ on $M$ is free then this map is a locally trivial fibration with fiber $EG$:
$$\begin{tikzcd}
  EG\arrow[hook,r] & M\times_G EG\arrow{d}{p_1}\\
  & M\slash G.
\end{tikzcd}$$
Since $EG$ is contractible, this map induces an isomorphism on (co)homology and thus:
\begin{theorem}[Cartan isomorphism]\label{Cartan}
  If the action of $G$ on $M$ is free, then we have isomorphisms
  $$p_{1_{\ast}}: H^G_\ast(M)\xrightarrow{\sim}H_{\ast}(M\slash G), \hspace{2mm} p_1^{\ast}:H^{\ast}(M\slash G)\xrightarrow{\sim}H^{\ast}_{G}(M).$$
\end{theorem}
However, the space $EG$ is usually an infinite-dimensional CW-complex, hence ill-suited for Morse-theoretical methods.

\subsubsection{Morse pushforwards}
\label{PushforwardSection}
The approach of this section is based on \cite{Kronheimer-Mrowka}.

Let $M,N$ be smooth manifolds, and $\varphi:M\to N$ a smooth map. Equip $M$ and $N$ with Morse-Smale pairs $(f,g)$ and $(h,g')$. We define a \emph{grafted Morse line} as a tuple $(\gamma^{-},\gamma^{+})$ where $\gamma^{-}:(-\infty,0]\to M$ and $\gamma^{+}:[0,\infty)\to N$ are smooth maps such that:
$$\begin{cases}
  \dot{\gamma}^{-}=-\nabla_g f(\gamma^{-}), & \dot{\gamma}^{+}=-\nabla_{g'} h (\gamma^{+}),\\
  \varphi(\gamma^{-}(0))=\gamma^{+}(0).
\end{cases}$$
Note that $\gamma^{-}$ converges to a critical point of $f$ as $t\to-\infty$ and $\gamma^{+}$ converges to a critical point of $h$ as $t\to +\infty$. Denote by $\mathcal{MG}_{\varphi}(x,y)$ the set of grafted Morse lines such that: $$\lim\limits_{t\to -\infty}\gamma^{-}(t)=x\text{ and }\lim\limits_{t\to \infty}\gamma^{+}(t)=y.$$ Elements of this set are pictured as in Figure \ref{fig:grafted-morse}.
\begin{figure}[h]
  \centering
  \begin{tikzpicture}
    \draw[blue,postaction={decorate,decoration={
        markings,
        mark=between positions 0.55 and 1 step 0.5 with {\arrow{>[scale=1.3,line width=.8pt]};},
      }}] (0,0)node[black,below]{$x$}node[fill=black,inner sep=1pt]{}--(5,0);
    \draw[red,postaction={decorate,decoration={
        markings,
        mark=between positions 0.55 and 1 step 0.5 with {\arrow{>[scale=1.3,line width=.8pt]};},
      }}] (5,0)--(10,0)node[black,below]{$y$}node[fill=black,inner sep=1pt]{};
    \draw (5,0.25)node[above]{$\varphi$}--(5,-0.25);
    \path (2.5,0.2)node[above]{$-\nabla_{g}f$}--(7.5,0.2)node[above]{$-\nabla_{g'}h$};
  \end{tikzpicture}
  \caption{A grafted Morse line in $\mathcal{MG}_{\varphi}(x,y)$.}
  \label{fig:grafted-morse}
\end{figure}
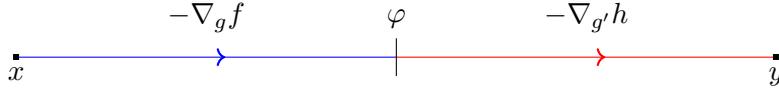

There is a bijection:
$$\mathcal{MG}_{\varphi}(x,y)\cong W^{u}(x)\times W^{s}(y)\cap \Gamma(\varphi)\subset M\times N$$
where $\Gamma(\varphi)$ is the graph of $\varphi$.
\begin{prop}
  If $\varphi(\crit f)\cap \crit(h)=\emptyset$, then there exists an open dense subset of $\operatorname{Met}(M)\times \operatorname{Met}(N)$ (the spaces of Riemannian metrics) such that for any metric in this subset the intersection above is transverse, and $\mathcal{MG}_{\varphi}(x,y)$ is a smooth manifold of dimension $\ind_f(x)-\ind_{h}(y)$.

  Furthermore, when this difference is one, this moduli space admits a compactification whose boundary is the disjoint union of:
  $$\mathcal{M}(x,z)\times \mathcal{MG}_{\varphi}(z,y)\text{ and }\mathcal{MG}_{\varphi}(x,w)\times \mathcal{M}(w,y),$$
  where the disjoint union is over $z\in \crit f$ with $\ind_{f}(x)-\ind_{f}(z)=1$, and over $w\in \crit h$ with $\ind_{h}(w)-\ind_{h}(y)=1$. Here we recall $\mathcal{M}(p,q)$ is the space of unparametrized gradient flow lines from $p$ to $q$.
\end{prop}
Given $f,h$ such that $\varphi(\crit f)\cap \crit (h)=\emptyset$ and a regular pair $(g,g')$ as in the proposition, one can form the Morse complexes of $f$ and $h$. The pushforward $\varphi_{\ast}: MC_{\ast}(M,f)\to MC_{\ast}(N,h)$ is defined as the linear extension of the map:
\begin{eqnarray*}
  \varphi_{\ast}:MC(M,f)&\to & MC(N,h)\\
  p & \mapsto & \sum_{\stackrel{q}{\ind_{f}(p)=\ind_{h}(q)}} \#_2 \mathcal{MG}_{\varphi}(p,q)\cdot q.
\end{eqnarray*}
The compactification of the 1-dimensional component of the manifolds $\mathcal{MG}_{\varphi}$ serve to show this is a chain map, and so there is an induced map $$\varphi_{\ast}:MH_{\ast}(M)\to MH_{\ast}(N)$$ which is independent of the auxiliary choices. Furthermore, as explained in \cite{Kronheimer-Mrowka}, after identifying Morse homology to singular homology, this map agrees with the singular pushforward.

\subsubsection{Equivariant Morse homology}\label{MorseData}

As mentioned, the space $EG$ being infinite dimensional makes it difficult to apply Morse-theoretical methods. However, we can always find a \emph{finite-dimensional approximation} of $EG$, in the following sense:
\begin{prop}
  If $G$ is a compact Lie group, then there exists a sequence $\{EG_n\}_{n\in \mathbb{N}}$ of closed smooth manifolds with maps $\iota_n:EG_n\to EG_{n+1}$ such that:
  \begin{itemize}
  \item each $EG_n$ is equipped with a left free $G$-action,
  \item each $\iota_n:EG_n\hookrightarrow EG_{n+1}$ is a $G$-equivariant embedding,
  \item the space $EG_{n+1}$ is obtained from $EG_{n}$ by attaching higher dimensional cells and $\lim\limits_{\to}EG_n\cong EG$. In particular:
    $$\lim\limits_{\to}H_{\ast}(EG_n)=\lim\limits_{\leftarrow} H^{\ast}(EG_n)\cong \begin{cases}
      \Z_2 & \ast=0\\
      0 & \text{otherwise}.
    \end{cases}$$
  \end{itemize}
\end{prop}
\begin{proof}
  Note that if $G\subset H$, then a model of $EH$ is a model of $EG$, since $EH$ has a free $G$-action and is contractible (however, the classifying spaces may of course be different).

  Every compact Lie group $G$ embeds into some $U(k)$ for $k$ sufficiently large (this is a consequence of the Peter-Weyl theorem, see \cite{Knapp} theorem 4.20 for example), so we can reduce to the case where $G=U(k)$, but an finite dimensional approximation exists for this group: these are the Stiefel manifolds.
\end{proof}
If $M$ is a closed manifold with a left $G$-action, we can form the sequence of closed manifolds $\{M\times_G EG_n\}_{n\in \mathbb{N}}$ where the $EG_n$ are as above (note that they are indeed smooth manifolds since the diagonal action is free, because it is free on $EG_n$). Furthermore, the map $\Id\times \iota_n: M\times EG_n\to M\times EG_{n+1}$ is $G$-equivariant and thus descends to the quotient, we denote the resulting maps as $\Id\times_G \iota_n:M\times_G EG_n\to M\times_G EG_{n+1}$.

We equip such a sequence of spaces and maps with \emph{universal Morse data} as follows: Since each $EG_{n+1}$ is obtained from $EG_{n}$ by attaching higher dimensional cells, there exists a sequence of Morse functions $f_n: M\times_G EG_n\to \R$ such that $\Id\times_G \iota_n (\crit f_n)\cap \crit f_{n+1}=\emptyset$. Henceforth we fix such a family of functions.

Given two consecutive manifolds of the sequence, there is a subset:
$$R(n,n+1)\subset \operatorname{Met}(M\times_G EG_n)\times \operatorname{Met}(M\times_G EG_{n+1})$$
consisting of pairs of metrics that are Morse-Smale and regular for the pushforward (i.e.\@ the involved moduli spaces are transversely cut-out). This subset is open and dense.

The space $R=\prod_{n\in \mathbb{N}} \operatorname{Met}(M\times_G EG_n)$ is a completely metrizable space (being a countable product of completely metrizable spaces): it is thus a Baire space. Set
$$V_n:= R(n,n+1)\times \prod_{j\neq n,n+1}\operatorname{Met}(M\times_G EG_j)\subset R.$$
They are open and dense subsets of $R$, and so $\bigcap_{n\in \mathbb{N}}V_n$ is a set of the second Baire category, thus dense in $R$.

This means that there are generic sequences of metrics such that all the moduli spaces involved in each differential and each pushforward are all smooth manifolds of the expected dimension \emph{at the same time}.

Define the \emph{equivariant Morse complex} as:
$$MC_{\ast}^{G}(M):= \overline{\operatorname{Tel}}(MC_{\ast}(M\times_G EG_n), \Id\times_G\iota_{n_\ast}).$$

Its homology, called \emph{equivariant Morse homology} and denoted $MH_{\ast}^{G}(M)$, satisfies:
$$MH_{\ast}^{G}(M)\cong \lim\limits_{\to}MH_{\ast}(M\times_G EG_n)\cong H_{\ast}^{G}(M)$$
where the first isomorphism comes from the natural isomorphism between the homology of a filtered homotopy colimit and the homology of the usual limit.

\section{Equivariant Lagrangian Floer homology}
In this section we construct equivariant Lagrangian Floer homology following \cite{Cazassus}. Our main contribution here is adapting the construction to the clean-intersection case, along with taking care of Novikov coefficients.
\subsection{Set-up}
We assume $G$ is a compact Lie group acting Hamiltonianly on $(X,\omega)$, with moment map $\mu_X:X\to \mathfrak{g}^{\ast}$, and that $L_0,L_1\subset\mu^{-1}(0)\subset X$ are $G$-invariant Lagrangians in clean intersection. We assume that $L_0$ and $L_1$ are \emph{relatively exact}, see (\ref{RE}), and that $X$ is either compact or $G$-convex at infinity:
\begin{equation}\label{GC}\tag{GC}
  \begin{split}	\text{There exists a convex pair on }X, (J_C,\rho)\textit{ such that }J_C \text{ and }\rho \\
    \text{are }G\text{-invariant and }d\rho(J_{C}X_{\xi})=\langle\mu_X,\xi\rangle
    \text{ for all }\xi\in \mathfrak{g}.
  \end{split}
\end{equation}

We note that the condition that both $L_0$ and $L_1$ are contained in $\mu_X^{-1}(0)$ is not too restrictive: indeed, two connected and $G$-invariant Lagrangians are contained in some level set of the moment map, which up to a shift in the moment map (if they both lie on the same level set) we can assume is the zero level.

\subsection{Symplectic Borel spaces}
Fix a finite-dimensional approximation of $EG$, say $\{EG_n\}_{n\in \mathbb{N}}$ with maps $\iota_n: EG_n\hookrightarrow EG_{n+1}$. By example~\ref{ExCot}, the cotangent bundle $T^{\ast}EG_n$ with its canonical symplectic form $\omega_{EG_n}$ is endowed with a Hamiltonian $G$-action with moment map $\mu_{EG_n}$. By Remark~\ref{RemarkDiag}, the diagonal action on the symplectic manifold $(X\times T^{\ast}EG_n, \omega+\omega_{EG_n})$ is Hamiltonian with moment map $\mu_{n}= \mu_{X}+\mu_{EG_n}$. Since the action of $G$ on $EG_n$ is free, this action is also free, and we can thus form the Marsden-Weinstein reduction:
$$X_{n}:= (X\times T^{\ast}EG_n)\slash\slash G.$$
By the Marsden-Weinstein theorem, this is a symplectic manifold, and we denote its induced symplectic form simply by $\omega_n$. In it sit two distinguished Lagrangians, namely $L_j^{n}:= L_j\times_G 0_{EG_n}\subset X_n$ for $j=0,1$.
\begin{prop}
  If $L_0,L_1$ are relatively exact, then so are $L_0^{n}$ and $L_1^{n}$.
\end{prop}
\begin{proof}
  Indeed, if $u:(D,\partial D)\to (X_n, L_j^{n})$ is a disc, then, by virtue of it being contractible, one can lift it to a disc $\tilde{u}:(D,\partial D)\to (X\times T^{\ast}EG_n, L_j\times 0_{EG_n})$ with the same symplectic area. Since $L_j$ is relatively exact and so is $0_{EG_n}$ we have:
  $$\int_{D}u^{\ast}\omega_{n_{\operatorname{red}}}=\int_{D} \tilde{u}^{\ast}(\omega+\omega_{EG_n})=0,$$
  and the proof is complete.
\end{proof}
We now exhibit a fibration analogous to $p_2:M\times_G EG\to BG$ in the classical case. This corresponds to \cite[Proposition 4.7]{Cazassus}, and we repeat the construction for completeness.

Equip $EG_n$ with a $G$-invariant metric. The its tangent space splits as:
$$T_{q}EG_n\cong T_{q}\mathcal{O}_{q}\oplus T_{q}\mathcal{O}_q^{\bot}$$
where $\mathcal{O}_{q}\cong G$ is the $G$-orbit of $q$. In particular, there is a canonical identification $\mathfrak{g}\cong T_q \mathcal{O}_{q}$ by $\xi \mapsto X_{\xi}(q)$. Taking duals yields a $G$-invariant splitting:
$$T_{q}^{\ast}EG_{n}\cong (T_q \mathcal{O}_{q})^{\ast}\oplus (T_{q}\mathcal{O}_q)^{\bot^{\ast}}\cong \mathfrak{g}^{\ast}\oplus \operatorname{Ann}_{q}(\mathfrak{g})$$
where $\operatorname{Ann}$ denotes the annihilator. Under this identification, $\mu_{EG_{n}}^{-1}(0)$ is identified with the vector sub-bundle:
$$\mu_{EG_{n}}^{-1}(0)= \bigcup_{q\in EG_n} \operatorname{Ann}_{q}(\mathfrak{g}).$$
Note that the orthogonal projection $T^{\ast}EG_n\to \mu_{EG_n}^{-1}(0)$ is $G$-equivariant since the splitting is $G$-invariant. We precompose this map with the second coordinate projection $X\times T^{\ast}EG_{n}\to T^{\ast}EG_n$ to get a $G$-invariant map:
$$\mu_{n}^{-1}(0)\to T^{\ast}EG_n\to \mu_{EG_n}^{-1}(0)$$
which descends to the quotients to yield:
$$\pi_2: X_n\to T^{\ast}BG_n \cong T^{\ast}EG_n\slash\slash G .$$
We can parametrize the fiber over a point $(q,p)$ as follows: pick a lift $(\tilde{q},\tilde{p})\in \mu_{EG_n}^{-1}(0)$ (there are $G$-many choices of such lifts), note that $\tilde{p}$ is an element of $\operatorname{Ann}_{\tilde{q}}(\mathfrak{g})$. Now set:
\begin{eqnarray*}
  \beta: X & \hookrightarrow & \pi_{2}^{-1}(q,p)\subset X_n\\
  x & \mapsto & [x, (\tilde{q},\tilde{p}-\mu(x))]
\end{eqnarray*}
where $\mu(x)$ is identified with an element of $\mathfrak{g}^{\ast}$ by virtue of the aforementioned splitting. Now we compute:
\begin{align*}  \beta^{\ast}(\omega+\omega_{EG_n})_{\operatorname{red}}(v,w)&=\omega(v,w)+\omega_{EG_n}((0,-d\mu(v)),(0,-d\mu(w)))\\                                                              &= \omega(v,w)
\end{align*}
since the cotangent fiber in $T^{\ast}EG_n$ is Lagrangian. Using this it is not hard to verify the following:
\begin{prop}
  The map $\pi_2:X_n \to T^{\ast}BG_n$ is a locally trivial fibration with structure group $G$ (which in particular acts by symplectomorphisms) and fibers symplectomorphic to $X$. Furthermore this fibration restricts to a fibration on the Lagrangians $L_j\hookrightarrow L_j \times_G 0_{EG_n}\to 0_{BG_n}$ (this is just the standard projection $p_2$ we defined in the previous sections).
\end{prop}

Now define the tangent distribution:
$$H_{[x,(\tilde{q},\tilde{p})]}:= (T_{[x,(\tilde{q},\tilde{p})]}\pi_{2}^{-1}(q,p))^{\omega}\subset T_{[x,(\tilde{q},\tilde{p})]}$$
that is, $H$ is the \emph{symplectic orthogonal to the fiber} (here we use the notation $\_ ^{\omega}$ even though it is the orthogonal with respect to $(\omega+\omega_{EG_{n}})_{\mathrm{red}}$).
\begin{prop}
  The distribution $H$ is a \emph{symplectic connection} in the following sense. The projection: $$\pi_{2_{\ast}}\rvert_{H_{[x,(\tilde{q},\tilde{p})]}}: H_{[x,(\tilde{q},\tilde{p})]}\to T_{(q,p)}T^{\ast}BG_n$$
  is a linear symplectomorphism.
\end{prop}
\begin{proof}
  Taking $(\tilde{q},\tilde{p})$ a lift of $(q,p)$ in $\mu_{EG_n}^{-1}(0)$, the tangent space splits as:
  $$T_{(\tilde{q},\tilde{p})}\mu_{EG_n}^{-1}(0)\cong T_{\tilde{q}}\mathcal{O}_{\tilde{q}}\oplus (T_{\tilde{q}}\mathcal{O}_{\tilde{q}})^{\bot}\oplus T_{\tilde{p}}\operatorname{Ann}_{\tilde{q}}(\mathfrak{g})\cong \mathfrak{g}\oplus (\mathfrak{g})^{\bot}\oplus \operatorname{Ann}(\mathfrak{g}).$$
  Meanwhile the tangent space $T_{(\tilde{q},\tilde{p})}T^{\ast}EG_n$ splits as:
  $$T_{(\tilde{q},\tilde{p})}T^{\ast}EG_n\cong \mathfrak{g}\oplus (\mathfrak{g})^{\bot}\oplus\mathfrak{g}^{\ast}\oplus \operatorname{Ann}(\mathfrak{g}).$$
  With these decompositions, the tangent map of the projection $T^{\ast}EG_n\to \mu_{EG_n}^{-1}(0)$ is just the projection in the appropriate components. Meanwhile a tangent vector to the fiber is represented uniquely in its equivalence class as
  $$\tilde{w}=[w, (0,-d\mu_X(w))]$$
  where $v\in T_{x}X$. Hence, if an element $\overline{v} \in H_{[x,(\tilde{q},\tilde{p})]}$ written as:
  $$\overline{v}=[v, (\alpha,\beta-d\mu_{X}(v))],$$
  it must be that $v=0$ (since it is orthogonal to the fiber), so it is uniquely represented as an element:
  $$\overline{v}=[0, (\alpha,\beta)]$$
  where $\beta\in \operatorname{Ann}(\mathfrak{g})$. Now decompose $\alpha$ as $\alpha=(\alpha_{\mathfrak{g}},\alpha^{\bot})\in \mathfrak{g}\oplus (\mathfrak{g})^{\bot}$ and observe that $\alpha_{\mathfrak{g}}$ has to be zero: indeed if it was not, $\overline{v}$ would be equivalent to
  $$\overline{v}\sim [-X_{\alpha_\mathfrak{g}}, (\alpha^{\bot}, \beta)]$$
  which is not orthogonal to the fiber. Hence $\overline{v}$ can be uniquely represented as $[0,(\alpha^{\bot}, \beta)]$ where $(\alpha^{\bot},\beta)\in (\mathfrak{g})^{\bot}\oplus\operatorname{Ann}(\mathfrak{g})$. Using this description, the proposition follows easily.
\end{proof}
Any compatible almost complex structure $J_{BG_n}$ on $T^{\ast}BG_n$ lifts to a compatible almost complex structure $\pi_{2}^{\ast}J_{BG_n}$ on the symplectic vector bundle $H\subset TX_n$. The vertical bundle $V\subset TX_n$ is symplectic (since the inclusion of the fiber is a symplectomorphism) and so it admits compatible almost complex structure $J_{V}$. We say that an almost complex structure on $X_n$ is \emph{split} if it is of the form:
$$J_{V}\oplus \pi_{2}^{\ast}J_{BG_n}\in \operatorname{End}(V\oplus H)=\operatorname{End}(TX_n).$$
Since the space of compatible almost complex structures on a symplectic vector bundle is non-empty and contractible, the space of split almost complex structures is also non-empty and contractible.

There is a class of almost complex structures that are naturally split: if $J_{X}$ is $G$-invariant, then we can unambiguously define the almost complex structure $J_{X}\oplus \pi_{2}^{\ast}J_{BG_n}$ on $X_n$. In particular, if $J_C$ is a $G$-convex structure, then $J_{C}\oplus \pi_{2}^{\ast}J_{BG_n}$ is well defined. Then:
\begin{prop}
  If $X$ is $G$-convex at infinity, then $X_n$ is convex at infinity.
\end{prop}
\begin{proof}
  Equip $X\times T^{\ast}EG_n$ with the almost complex structure $J_{C}\oplus J_{EG_n}$ where $J_{EG_n}$ is induced by some Riemannian metric. $T^{\ast}EG_n$ has a $G-$convex function $\rho_{EG_n}$ given by $(q,p)\mapsto \norm{p}^2$. Now $\rho\oplus\rho_{EG_n}:X\times T^{\ast}EG_n\to \R_{\geq 0}$ is a $G$-invariant convex pair. Furthermore, restricted to $\left(\mu_{X}+\mu_{EG_n}\right)^{-1}(0)$ it satisfies $d\rho(J_{C}X_{\xi})+d\rho_{EG_n}(J_{EG_n}X_{\xi})=0$. The claim follows from the following lemma. Furthermore, the resulting almost complex structure $\tilde{J_C}$ is split. We denote by $\tilde{\rho}$ the induced plurisubharmonic function.
\end{proof}
As a corollary:
\begin{lema}\label{LemmaConvex}
  Suppose $Y$ is a Hamiltonian $G$-manifold with a $G$-invariant convex structure $(\rho,J_C)$ such that $d\rho\rvert_{\mu_{Y}^{-1}(0)}(J_{C}X_{\xi})=0$ for all $\xi \in \mathfrak{g}$. Then the quotient pair $(\overline{J_C},\overline{\rho})$ is a convex structure on $Y\slash \slash G$.
\end{lema}
We also observe that if the pair $L_0,L_1$ satisfies \eqref{TH}, so does the pair $L_0^{n},L_1^{n}$ inside $X_n$, so we can reduce to distinct discs in our moduli spaces.

In the following, we denote by $\mathcal{J}_{\operatorname{adm}}^{\operatorname{split}}(X_n)$ the space of split almost complex structures that agree with $\tilde{J_C}$ outside a compact set.
\begin{theorem}
  There exists a subset of the second Baire category in the set of smooth paths $\mathcal{J}_{\operatorname{adm}}^{\operatorname{split}}(X_n)$ such that for all $J_t$ in this subset the moduli spaces involved in the cascade Floer differential (of $X_n$ with Lagrangians $L_j^{n}$) are transversely cut-out.
\end{theorem}
\begin{proof}
  We start with a simpler statement: we will show the claim for the space of a single Floer strip $u$. We follow essentially the same steps as in \cite[Section 6]{Schmaschke}: given connected components $C_{-},C_{+}$ of $L_0^{n}\cap L_1^{n}$, consider the Banach manifold $\mathcal{B}^{1,p}_{\delta}(C_{-},C_{+})$ of maps $u:\Theta\to X_n$ of regularity $W^{1,p}_{\delta}$. The subscript $\delta$ means $u$ has $\delta$-fast exponential convergence as $s\to\pm \infty$ to tangent vectors in $C_{\pm}$, and consider the Banach bundle $\mathcal{E}^{p}_{\delta}\to \mathcal{B}^{1,p}_{\delta}(C_{-},C_{+})$ whose fiber over $u$ is $L^{p}_{\delta}(\Gamma(u^{\ast}TX_n\otimes \Omega^{0,1}_{\Theta}))$ the space of $L^{p}$ sections with exponential decay. The Cauchy-Riemann operator is a section of this bundle and the linearization is Fredholm for a sufficiently small $\delta$.

  The only change in the proof is to show that the universal moduli space
  \begin{equation*}
    \tilde{\mathcal{M}}(C_{-},C_{+}, \mathcal{J}_{\operatorname{adm}}^{\operatorname{split}}):=\left\{
      (u,J_t)\mid\begin{aligned}
        &u\text{ is a }J_t\text{-holomorphic Floer strip},\\
        &\lim_{s\to\pm \infty} u(s,t)\in C_{\pm},\text{ and},\\
        &J_t\in \mathcal{J}_{\operatorname{adm}}^{\operatorname{split}}(X_n)
      \end{aligned}\right\}
  \end{equation*}
  is a Banach manifold. To see this, it suffices to show that at $(u,J_{t})\in \tilde{\mathcal{M}}(C_{-},C_{+}, \mathcal{J}_{\operatorname{adm}}^{\operatorname{split}})$ the operator
  \begin{eqnarray*}
    D^{\operatorname{univ}}_{u,J_t}: T_{u}\mathcal{B}^{1,p}_{\delta}(C_{-},C_{+})\oplus T_{J_t}\mathcal{J}_{\operatorname{adm}}^{\operatorname{split}}(X_n)&\to& \mathcal{E}^{p}_{\delta,u}\\
    (\xi, Y) &\mapsto & D_{u,J_t}(\xi)+  Y(\partial_t u)
  \end{eqnarray*}
  is surjective. The tangent space $T_{u}\mathcal{B}^{1,p}_{\delta}$ is described as sections of $u^{\ast}TX_n$ of regularity $W^{1,p}$ and $\delta$-fast exponential convergence.

  Note that, since $\pi_2:X_n\to T^{\ast}BG_n$ is holomorphic for any split almost complex structure, then $\pi_{2}(u)$ is a pseudoholomorphic strip with boundary in the zero-section, so it is constant and the image of $u$ lies entirely in a fiber, symplectomorphic to $X$.

  We split $u^{\ast}TX_n$ using $TX_n=V\oplus H$ where $H$ is the aforementioned horizontal distribution, and so $u^{\ast}TX_n\cong u^{\ast}V\oplus u^{\ast}H$. Note that we can trivialize $u^{\ast}H$ so that it is isomorphic to the trivial bundle $u^{\ast}H\cong u^{\ast}T_{\pi_2(u)}T^{\ast}BG_n$. With this splitting and identifying the fiber with $X$ we can write
  $$T_{u}\mathcal{B}^{1,p}_{\delta}(C_{-},C_{+})\cong T_{u'}\mathcal{B}^{1,p}_{\delta}(X,C_{-}', C_{+}')\oplus T_{\pi_{2}(u)}\mathcal{B}^{1,p}_{\delta}(T^{\ast}BG_n, 0_{BG_n}, 0_{BG_n})$$
  where $u'$ is the fiber part of $u$, and $C_{\pm}^{'}$ denote the connected components of $L_0\cap L_1\subset X$. Similarly, we split
  $$\mathcal{E}^{p}_{\delta,u}\cong \mathcal{E}^{p}_{\delta,u'}\oplus \mathcal{E}^{p}_{\delta, \pi_2(u)}$$
  and the tangent space at $J_t=J_t'\oplus J_t''$ of  $T_{J_t}\mathcal{J}_{\operatorname{adm}}^{\operatorname{split}}(X_n)$ as $T_{J_t'}\mathcal{J}_{\operatorname{adm}}(X)\oplus T_{J_t''}\mathcal{J}_{\operatorname{adm}}(T^{\ast}BG_n)$. With these splittings in mind, the universal operator takes the form
  $$D^{\operatorname{univ}}_{u,J_t}: (\xi, \xi', Y,Y')\mapsto (D_{u',J_t'}(\xi)+ Y(\partial_t u'), D_{\pi_2(u),J_t''}(\xi')).$$
  Now as in \cite{Schmaschke}, section 6.4, the existence of regular points implies that the first component of this map is surjective. To show that the second factor is surjective it is enough to see that constant Floer strips on $T^{\ast}BG_n$ with boundary on $0_{BG_n}$ are transversely cut-out, but this follows from an index computation, as the space of constant solutions has dimension $\dim BG_n$ and the index is $\operatorname{Mas}(\pi_2(u))+\dim BG_n=\dim BG_n$. Alternatively this is shown in \cite{Lekili-Lipyanskiy}, section 2.4.

  The rest of the proof follows the same steps as \cite{Schmaschke} section 6 with this fact in mind.

  \textbf{Note}: We cheated a bit, in the sense that $\mathcal{J}_{\operatorname{adm}}^{\operatorname{split}}(X_n)$ is \emph{not} a Banach manifold. To get around this we use the Floer $\varepsilon$-norm as in \cite{Schmaschke} section 6.5 to find a subspace which is a separable Banach manifold. The proof goes through without any other changes.
\end{proof}
Thanks to this result, we can achieve transversality for the space of the Floer differential using split structures, and for each $n$ we get a Floer chain complex $FC_{\ast}(L_0^{n},L_1^{n})$.
\subsection{Increment correspondence and pushforward}
Following \cite{Cazassus}, we construct here a Lagrangian correspondence $\mathcal{I}_n\subset X_n\times\overline{X_n}$, which, combined with quilts, give rise to chain maps $FC_{\ast}(L_0^{n},L_1^{n})\to FC_{\ast}(L_0^{n+1},L_1^{n+1})$. These will be used for the telescope construction.

Given a map $\iota_n:EG_n\hookrightarrow EG_{n+1}$ of the fixed finite-dimensional approximation, we consider the correspondence
$$N_{\Gamma(\iota_n)}:= \{((q,p),(q',p'))\in T^{\ast}EG_n\times T^{\ast}EG_{n+1}\mid q'=\iota_n(q), p=\iota_{n}^{\ast}p'\}$$
and define a Lagrangian correspondence
$$\widehat{\mathcal{I}_n}:= \Delta_{X}\times N_{\Gamma(\iota_n)}\subset X\times T^{\ast}EG_n\times\overline{X\times T^{\ast}EG_{n+1}}.$$
Note that it is $G\times G$-invariant where the action is given by
$$(g_1,g_2)\cdot ((x,(q,p)),(y,(q',p')))= ((g_1\cdot x,g_1\cdot (q,p)), (g_2\cdot y, g_2\cdot (q',p'))).$$
This action is Hamiltonian with moment map
\begin{eqnarray*}
  \Phi_n: X\times T^{\ast}EG_n\times\overline{X\times T^{\ast}EG_{n+1}} & \to & \mathfrak{g}^{\ast}\oplus \mathfrak{g}^{\ast}\\
  (x_1,x_2)&\mapsto & (\mu_n(x_1),-\mu_{n+1}(x_2)).
\end{eqnarray*}
Note that the sign of the second term is negative since the symplectic structure is $-\omega_n$. The correspondence $\widehat{\mathcal{I}_n}$ is a $G$-invariant Lagrangian with respect to the diagonal of the $G\times G$-action. As shown in \cite{Cazassus} section 4.3, the intersection $\widehat{\mathcal{I}_n}\cap \Phi^{-1}(0)$ is clean, and we define
$$\widetilde{\mathcal{I}_n}:= G\times G\cdot\left( \widehat{\mathcal{I}_n}\cap \Phi^{-1}(0)\right).$$
Finally, set
$$\mathcal{I}_n:= \widetilde{\mathcal{I}_n}\slash (G\times G).$$
Define the quilted surface $\underline{Z}$ as follows:
\begin{itemize}
\item The total space is $Z=\Theta\setminus\{0,i\}= \{s+it\in \C\mid t\in [0,1]\}\setminus\{0,i\}$,
\item there is only one seam, given by $S=\{it\mid t\in (0,1)\}$,
\item there is one incoming end at $s\to-\infty$, one outgoing end as $s\to +\infty$ and at the points $0,i$ we have two free ends, which we will simply assign to be of the form $\R_{>0}\times [0,\delta]$ for some $\delta<1$.
\end{itemize}
The surface has two patches, $P_0$ and $P_1$ given by the left and right sides of the seam respectively, and four true boundary components
$$
\begin{cases}\partial^{j}P_0= \{s+it\in P_0\mid t=j \}, &j=0,1,\\
  \partial^{j}P_1=\{s+it\in P_1\mid t=j\}, & j=0,1.
\end{cases}$$
We decorate the quilted surface $\underline{Z}$ by assigning the manifolds $X_n$ and $X_{n+1}$ to $P_0$ and $P_1$ respectively, the Lagrangian correspondence $\mathcal{I}_n$ to the seam $S$, the Lagrangian submanifolds $L_j^{n}$ to $\partial^{j}P_0$ and $L_j^{n+1}$ to $\partial_j P_1$. We draw the following picture:
\begin{figure}[h]
  \centering
  \begin{tikzpicture}
    \fill[cyan!20!white] (0,0) rectangle +(5,1);
    \fill[red!20!white] (5,0) rectangle +(5,1);
    \draw (0,1) -- (10,1);
    \draw (0,0) -- (10,0);
    \draw (5,0) -- (5,1)node[above]{$\mathcal{I}_{n}$};
    \path (2.5,0.5)node{$X_{n}$}--(7.5,0.5)node{$X_{n+1}$};
    \path (2.5,-0.3)node{$L^{n}_{0}$}--+(5,0)node{$L^{n+1}_{0}$};
    \path (2.5,1.3)node{$L^{n}_{1}$}--+(5,0)node{$L^{n+1}_{1}$};
    \path[every node/.style={draw,circle,fill=white,inner sep=1pt}] (5,1)node{}--(5,0)node{};
  \end{tikzpicture}
  \caption{The decorated quilted surface $\underline{Z}$.}
  \label{fig:dec-quilt-surf}
\end{figure}
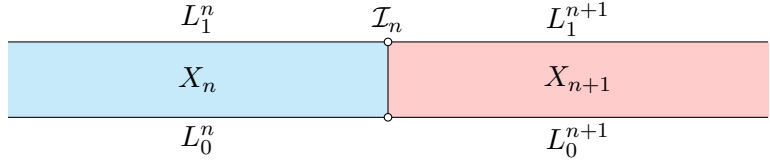

We now associate a subclass of quilted perturbation data to $\underline{Z}$:
\begin{itemize}
\item The quilted almost complex structures take values in split almost complex structures $J_{P_0}:P_0\to \mathcal{J}^{\operatorname{split}}_{\operatorname{adm}}(X_n)$, $J_{P_1}:P_1\to \mathcal{J}^{\operatorname{split}}_{\operatorname{adm}}(X_{n+1})$, and are asymptotic on the incoming and outgoing ends to $J_t$ and $J_t'$.
\item The quilted Hamiltonian terms are compactly supported in the interior of each patch.
\end{itemize}
Denote by $\operatorname{Per}(\mathcal{I}_n)$ the class of such perturbation data and a given element by $(\underline{J},\underline{H})$. Introduce $\mathcal{Q}(\mathcal{I}_n, C_{-},C_{+})$ as the set of pseudoholomorphic quilts $\underline{u}=(u_1,u_2)$ such that:
\begin{enumerate}
\item $E(\underline{u})=E(u_1)+E(u_2)<\infty,$
\item  $\lim\limits_{s\to\pm\infty}\underline{u}(s,t)\in C_{\pm}$.
\end{enumerate}
We generally suppress the perturbation data from the notation, and only add it when relevant.
\begin{theorem}\label{TransversalityII}
  For a given $\underline{J}$, there exists a comeagre subset of Hamiltonian perturbations $\mathcal{H}$ such that for all $(\underline{J}, \underline{H})\subset \operatorname{Per}(\mathcal{I}_n)$ with $\underline{H}\in \mathcal{H}$ the moduli space $\mathcal{Q}(\mathcal{I}_n, C_{-},C_{+})$ is a smooth manifold of dimension:
  $$\operatorname{Mas}(\underline{u})+\frac{1}{2}\dim C_{-}+\frac{1}{2}\dim C_{+}-\frac{1}{2}(\dim BG_{n+1}-\dim BG_n)$$
  and the evaluation maps:
  \begin{eqnarray*}
    \operatorname{ev}_j: \mathcal{Q}(\mathcal{I}_n, C_{-},C_{+})&\to & C_{-} \slash C_{+}\\
    \underline{u}=(u_1,u_2)&\mapsto & u_1(-\infty)\slash u_2(+\infty)
  \end{eqnarray*}
  are submersive.
\end{theorem}
\begin{proof} As before, consider the Banach manifold of quilted maps $\mathcal{B}^{1,p}_{\delta}(C_{-},C_{+})$ of local regularity $W^{1,p}$, such that at each end we have $\delta$-fast exponential convergence (note that we also have to assure exponential convergence on the free ends). Here $\delta$ is a constant that only depends on the Lagrangians, $\mathcal{I}_{n}$ and the asymptotic almost complex structures of $\underline{J}$. Now the tangent space space at a given $\underline{u}=(u_1,u_2)$ is given by quilted $W^{1,p}_{\delta}$ sections of $u_{1}^{\ast}TX_{n}\oplus u_{2}^{\ast}TX_{n+1}$ that satisfy the boundary conditions and converge $\delta$-exponentially fast to tangent vectors on the ends. We denote it by $T_{\underline{u}}\mathcal{B}^{1,p}_{\delta}(C_{-},C_{+})$ and we denote by $T_{\underline{u}}^{0}\mathcal{B}^{1,p}_{\delta}(C_{-},C_{+})$ the subspace of all sections that exponentially converge to zero on the incoming and outgoing ends. Following the standard arguments, one also considers the Banach bundle $\mathcal{E}^{p}_{\delta}$ whose fiber over $\underline{u}$ is: $$L^{p}_{\delta}(\Gamma(u_{1}^{\ast}TX_n\otimes \Omega^{0,1}_{P_0}))\oplus L_{\delta}^{p}(\Gamma(u_{2}^{\ast}TX_{n+1}\otimes \Omega^{0,1}_{P_1})).$$

  Consider the space of Hamiltonian perturbations supported in a given open set $U=U_1\sqcup U_2$ where $U_{j}\subset \operatorname{int}(P_j)$ and denote it by $\mathcal{H}_{U}$. To prove the statement it suffices to show that the \emph{universal moduli space} given by
  $$\tilde{\mathcal{Q}}(C_{-},C_{+},\mathcal{I}_n,  \mathcal{H}_{U}):= \{ (\underline{u},\underline{H})\mid \underline{u}\in \mathcal{Q}(C_{-},C_{+}, \mathcal{I}_n;\underline{H}), \underline{H}\in \mathcal{H}_{U}\}$$
  is a Banach manifold. To see this we only need to show that the universal operator:
  \begin{eqnarray*}
    D^{\operatorname{univ}}_{\underline{u},\underline{J},\underline{H}}: T_{\underline{u}}\mathcal{B}^{1,p}_{\delta}(C_{-},C_{+})\oplus T_{\underline{H}}\mathcal{H}_{U}&\to & \mathcal{E}^{p}_{\delta,\underline{u}}\\
    (\underline{\xi},\underline{K})&\mapsto & D_{\underline{u},\underline{J},\underline{H}}(\underline{\xi})+ (\delta K)^{0,1}
  \end{eqnarray*}
  is surjective, where $(\delta K)^{0,1}$ denotes the fiberwise one-form with values in Hamiltonian vector fields obtained from $K$. Take $\underline{\eta}$ in the cokernel. It is, in particular, in the kernel of the formal adjoint of $D_{\underline{J},\underline{H}}\rvert_{T^{0}_{\underline{u}}\mathcal{B}^{1,p}_{\delta}}$ so by elliptic regularity it is smooth. Furthermore, as explained in \cite[Section 9k]{Seidel2}, we have that:
  $$\int_{P_0}\langle \eta_1, (\delta K_1)^{0,1}\rangle= \int_{P_1}\langle \eta_2, (\delta K_2)^{0,1}\rangle=0.$$
  But inside $U_1$ and $U_2$, the terms $(\delta K_i)^{0,1}$ can be chosen in an arbitrary way, so $\eta_i\rvert_{U_i}\equiv 0$. But analytic continuation implies that $\eta_i\equiv 0$ everywhere, hence the operator is surjective.

  We have shown something stronger: the operator $D^{\operatorname{univ}}_{\underline{u}, \underline{J},\underline{H}}$ is surjective when restricted to $T^{0}_{\underline{u}}\mathcal{B}^{1,p}_{\delta}(C_{-},C_{+})$. So to show that the evaluation map $\operatorname{ev}_0$ is submersive, we can take an arbitrary section pair $(\underline{\xi}_0, \underline{K}_0)$ in  $T_{\underline{u}}\mathcal{B}^{1,p}_{\delta}(C_{-}) \oplus T_{\underline{H}}\mathcal{H}_{U}$ such that $d\operatorname{ev}_0 (\underline{\xi}_0,\underline{K}_0)= v\in T_{\operatorname{ev}_0 (\underline{u})}C_{-}$. By surjectivity, there exists $(\underline{\xi}_1,\underline{K}_1)$ such that $	D^{\operatorname{univ}}_{\underline{u},\underline{J},\underline{H}}(\underline{\xi}_1,\underline{K}_1)=         D^{\operatorname{univ}}_{\underline{u},\underline{J},\underline{H}}(\underline{\xi}_0, \underline{K}_0)$ and so $(\underline{\xi}_0-\underline{\xi}_1, \underline{K}_0-\underline{K}_1)$ is in the kernel and satisfies that the evaluation at $-\infty$ is $v$. We proceed in an identical way for $\operatorname{ev}_1$. This idea is present in \cite[Theorem C.13]{Frauenfelder}.

  As is standard, the projection from the universal moduli space onto the second coordinate satisfies the hypotheses for the Sard-Smale theorem and so the set of regular values is of the second Baire category and for these regular values the spaces $\mathcal{Q}(\mathcal{I}_n, C_{-},C_{+})$ are transversely cut-out and have submersive evaluation maps.

  \textbf{Warning}: As done before, there is a slight gap in the argument: the space of Hamiltonian perturbations is not Banach but only Fréchet. One should replace it by a dense Banach submanifold using Floer's $\varepsilon$-norm, as pointed out in \cite[Section 9k]{Seidel2}. The proof goes through without any changes.

  \textbf{Fredholm index}: The index is independent of our choice of Hamiltonian perturbations and almost complex structures (since the operators are Fredholm-homotopic), so we fix a $G$-invariant almost complex structure $J^{\operatorname{ref}}$ on $X$ (these exist and form a contractible space) and consider two split almost complex structures $J^{\operatorname{ref}}\oplus \pi_{2}^{\ast}J_{BG_n}$ and $J^{\operatorname{ref}}\oplus \pi_{2}^{\ast}J_{BG_{n+1}}$ and we also set the Hamiltonian perturbations to zero. Take $\underline{u}$ a solution, then the projection $\pi_2(\underline{u})=(\pi_2(u_1),\pi_2 (u_2)): P_0\times P_1 \to T^{\ast}BG_n\times T^{\ast}BG_{n+1}$ is a pseudoholomorphic quilt with Lagrangian correspondence $N_{\Gamma(\overline{\iota_n})}$ in the middle (here $\overline{\iota_n}:BG_n\to BG_{n+1}$ is the induced map from $\iota_n:EG_n\to EG_{n+1}$): it must then have zero area and thus be constant.
  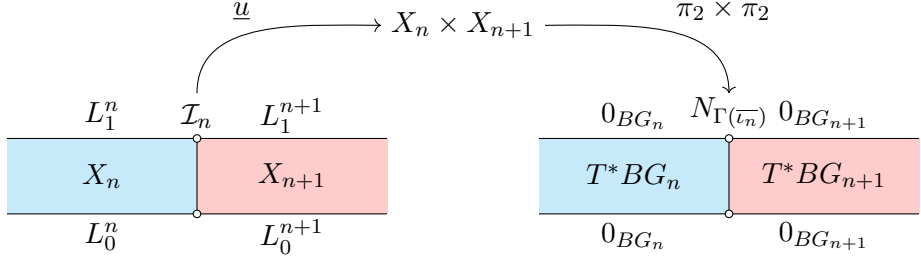
\begin{figure}[h]
    \centering
    \begin{tikzpicture}
      \begin{scope}[xscale=.5]
        \fill[cyan!20!white] (0,0) rectangle +(5,1);
        \fill[red!20!white] (5,0) rectangle +(5,1);
        \draw (0,1) -- (10,1);
        \draw (0,0) -- (10,0);
        \draw (5,0) -- (5,1)node[above]{$\mathcal{I}_{n}$};
        \path (2.5,0.5)node{$X_{n}$}--(7.5,0.5)node{$X_{n+1}$};
        \path (2.5,-0.3)node{$L^{n}_{0}$}--+(5,0)node{$L^{n+1}_{0}$};
        \path (2.5,1.3)node{$L^{n}_{1}$}--+(5,0)node{$L^{n+1}_{1}$};
        \path[every node/.style={draw,circle,fill=white,inner sep=1pt}] (5,1)node{}--(5,0)node{};
      \end{scope}
      \begin{scope}[shift={(7,0)},xscale=.5]
        \fill[cyan!20!white] (0,0) rectangle +(5,1);
        \fill[red!20!white] (5,0) rectangle +(5,1);
        \draw (0,1) -- (10,1);
        \draw (0,0) -- (10,0);
        \draw (5,0) -- (5,1)node[above]{$N_{\Gamma(\overline{\iota_{n}})}$};
        \path (2.5,0.5)node{$T^{*}BG_{n}$}--(7.5,0.5)node{$T^{*}BG_{n+1}$};
        \path (2.5,-0.3)node{$0_{BG_{n}}$}--+(5,0)node{$0_{BG_{n+1}}$};
        \path (2.5,1.3)node{$0_{BG_{n}}$}--+(5,0)node{$0_{BG_{n+1}}$};
        \path[every node/.style={draw,circle,fill=white,inner sep=1pt}] (5,1)node{}--(5,0)node{};        
      \end{scope}
      \path (2.5,1.5) +(3.5,1) node(A){$X_{n}\times X_{n+1}$};
      \draw[->] (2.5,1.6)to[out=90,in=180]node[above left]{$\underline{u}$}(A.west);
      \draw[->] (A.east)to[out=0,in=90] node[above right]{$\pi_{2}\times\pi_{2}$}(9.5,1.6);
    \end{tikzpicture}    
    \caption{The projection of $\underline{u}$ can be seen as a pseudoholomorphic quilt with the indicated boundary conditions, it is therefore constant.}
    \label{fig:project-quilt-constant}
  \end{figure}

  So $\underline{u}$ can be factored into the fibers $X\times X$, and the seam would now be decorated with the diagonal and the boundary would be decorated by $L_0$ and $L_1$ respectively. Since the linearized Cauchy-Riemann operator splits as explained before, we can compute each index separately.

  \textbf{Index on the fiber}: Since the seam is decorated by the diagonal, we may choose a complex structure that erases the seam completely and thus get a map from the standard Floer strip to $X$ with Lagrangian boundary conditions. The index is calculated in \cite{Schmaschke}, section 5 and is given by
  $$\ind D_{u'}^{\operatorname{fiber}}=\operatorname{Mas}(u')+\frac{1}{2}\dim C_{-}'+ \frac{1}{2}\dim C_{+}'$$
  where $u'$ is the identified map on the fiber and $C_{\pm}'$ denotes the connected component of the fiber of $C_{\pm}$.

  \textbf{Index on the base}: By a standard biholomorphism we may remove the singularities at $\pm\infty$ (see for example \cite{Oh2}) and reduce to the problem of quilted half-discs, which in turn can be folded to give strips on the product.
  
  By formula 5.8 in \cite{Schmaschke}, we know that the index of any such strip is
  $$\ind D_{\pi_2(\underline{u})}^{\operatorname{base}}=\operatorname{Mas}(\pi_2(\underline{u}))+\dim(0_{BG_n}\times 0_{BG_{n+1}}\cap N_{\Gamma(\overline{\iota_n})})= \dim BG_n.$$
  In summary, the index is given by
  \begin{align*}
    \ind D_{\underline{u}}&= \operatorname{Mas}(\underline{u})+\frac{1}{2}\dim C'_{-}+\frac{1}{2}\dim C_{+}'+\dim BG_n\\
                          &= \operatorname{Mas}(\underline{u})+\frac{1}{2}\dim C_{-}+\frac{1}{2}\dim C_{+}-\frac{1}{2}(\dim BG_{n+1}-\dim BG_{n}),
  \end{align*}
  using $\dim C_{-}= \dim C_{-}'+\dim BG_n$ and $\dim C_{+}=\dim C_{+}'+\dim BG_{n+1}.$ The dimension formula follows.
\end{proof}
\begin{obs}
  Note that the theorem holds with $\underline{J}$ \emph{fixed} and varying the Hamiltonian perturbations only.
\end{obs}
When the connected components $C_{-}$ or $C_{+}$ are irrelevant, we will simply write $\mathcal{Q}(\mathcal{I}_n)$ for the space of quilted strips as above.
\begin{cor}
  If $\underline{H}=0$, then the moduli space of constant $\mathcal{I}_n$-quilts is transversely cut-out and diffeomorphic to the graph
  $$\Gamma(\Id\times_G \iota_n)\subset (L_0^n\cap L_1^n)\times (L_0^{n+1}\cap L_1^{n+1}).$$
\end{cor}
\begin{proof}
  For a given connected component $C$ of $L_0\cap L_1$, a constant quilt in $C$ is just a pair $(p, \Id\times_G \iota_n(p))$, hence the space of constant quilts is diffeomorphic to the graph as stated. To see that it is transversely cut-out (i.e. of the expected dimension) we plug $C_{-}'=C_{+}'$ and $\operatorname{Mas}(\underline{u})=0$ in the above equation to get
  $$\ind D_{\underline{u}}= \dim C_{-}'+\dim BG_{n}=\dim C_{-}$$
  which is the dimension of the graph.
\end{proof}
The next step is to construct the pushforward between our Floer chain complexes. Once and for all, we pick a universal Morse datum for the manifolds $L_0^{n}\cap L_1^{n}$ as in subsection \ref{MorseData}. We denote the Morse functions by $f_n$ and the metrics by $g_n$.

Given two Floer complexes $FC_{\ast}(L_0^n,L_1^{n};\Lambda_0)$ and $FC_{\ast}(L_0^{n+1},L_1^{n+1};\Lambda_0)$ equipped with regular split almost complex structures $J_t$ and $J_t'$ and points $p\in \crit f_n, q\in \crit f_{n+1}$, a \emph{parametrized quilt cascade} with $m$ non-quilted jumps from $p$ to $q$ representing the homotopy class $A$ is a tuple: $$(u_1,\dots, u_{k}, \underline{u}, u_1',\dots, u_{l}')$$ with $k+l=m$ such that:
\begin{itemize}
\item Each $u_j$ is a \emph{nonconstant} $J_t$-holomorphic Floer strip in $X_n$, and each $u_{j}'$ is a \emph{nonconstant} $J_t'$-holomorphic Floer strip in $X_{n+1}$.
\item $\underline{u}$ is a $\mathcal{I}_n$-quilted strip as above, with perturbation data $(\underline{J},\underline{H})\in \operatorname{Per}(\mathcal{I}_n)$ where $\underline{J}$ is asymptotic to $J_t$ and $J_t'$ as $s\to\pm \infty$.
\item For $j\in \{1,\dots, k\}$ there exists $t_{j}\geq 0$ such that
  $$\begin{cases}
    \psi^{t_j}(u_j(+\infty))=u_{j+1}(-\infty), & j\neq k\\
    \psi^{t_k}(u_{k}(+\infty))=\underline{u}(-\infty).
  \end{cases}$$
  Here $\psi$ denotes the negative gradient flow with respect to the corresponding Morse data.
\item Similarly for $j'\in \{1,\dots, l-1\}$ there exist $t_j'\geq 0$ such that
  $$\begin{cases}
    \psi^{t_1'}(\underline{u}(+\infty))=u_1'(-\infty),\\
    \psi^{t_j'}(u_j'(+\infty))= u_{j+1}'(-\infty), & j'\neq 1.
  \end{cases}$$
\item $u_1(-\infty)\in W^{u}(p)$ and $u_{l'}\in W^{s}(q)$.
\item The $u_j$ and $u_j'$ each form a \emph{distinct tuple} (i.e. no one is a reparametrization of another).
\item The concatenation represents the homotopy class $A$.
\end{itemize}
The set of such quilt-cascades carries a natural $\R^{m}$ action by time-shift on each non-quilted component, and we denote the quotient by
$$\mathcal{QC}_{m}(p,q,A;\mathcal{I}_n).$$
Finally set:
$$\mathcal{QC}(p,q,A;\mathcal{I}_n)=\bigcup_{m\in \mathbb{N}} \mathcal{QC}_{m}(p,q,A;\mathcal{I}_n).$$
We picture quilt-cascades as follows:
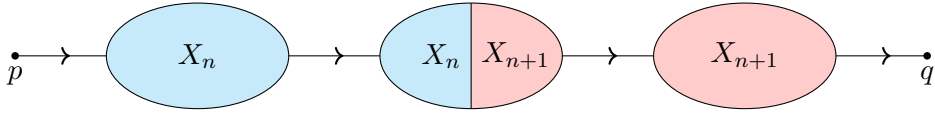
\begin{figure}[h]
  \centering
    \begin{tikzpicture}[xscale=1.2]
    \path[decorate,decoration={
      markings,
      mark=between positions 0.06 and 1 step 0.3 with {\arrow{>[scale=1.3,line width=.8pt]};},
    }] (0,0) to (10,0);
    \draw (0,0) node[below] {$p$} node[fill, inner sep=1pt, circle] {}-- (10,0)node[fill, inner sep=1pt, circle] {}node[below]{$q$};
    \path[fill=cyan!20!white] (2,0) circle (1 and 0.7) (5,0.7) arc (90:270:1 and 0.7)--cycle;
    \path[fill=red!20!white] (8,0) circle (1 and 0.7) (5,0.7) arc (90:-90:1 and 0.7)--cycle;
    \draw (5,0.7)--(5,-0.7);
    \draw (2,0) circle (1 and 0.7) node{$X_{n}$} (5,0) circle (1 and 0.7) node[right]{$X_{n+1}$}node[left]{$X_{n}$} (8,0) circle (1 and 0.7) node{$X_{n+1}$};
  \end{tikzpicture}
  \caption{A quilt cascade from $p$ to $q$ with two non-quilted jumps and one quilted jump.}
  \label{fig:quilt-cascades}
\end{figure}
\begin{theorem}
  If the almost complex structures are regular for each Floer complex and $H$ is regular for the space of quilted strips, then $\mathcal{QC}(p,q,A;\mathcal{I}_n)$ is a smooth manifold with corners. The zero-dimensional component of this manifold is compact, hence a finite set, and the $1$-dimensional component admits a compactification whose boundary is the disjoint union of:
  $$\bigsqcup_{\stackrel{z}{B+C=A}}\mathcal{M}_{[0]}(p,z,B)\times \mathcal{QC}_{[0]}(z,q,C;\mathcal{I}_n),$$ and: $$\bigsqcup_{\stackrel{z}{B+C=A}}\mathcal{QC}_{[0]}(p,z,B;\mathcal{I}_n)\times \mathcal{M}_{[0]}(z,q,C).$$
\end{theorem}
\begin{proof}
  The transversality part is clear: indeed using theorem \ref{TransversalityII} and using a fiber product description as in theorem \ref{TransversalityI} we readily get the result.

  Turning now to compactness, there is another possible source of lack of compactness corresponding to the quilted jump $\underline{u}$. Indeed, by Gromov-Compactness for quilts (see, e.g., \cite[Lemma 6]{Lekili-Lipyanskiy}) a sequence of quilts $\underline{u}$ may degenerate by:
  \begin{enumerate}
  \item Bubbling off at a true boundary point or at a non-seam interior point,
  \item strip breaking at the free ends,
  \item bubbling at a seam point (i.e. a quilted bubble),
  \item breaking at an incoming or outgoing end.
  \end{enumerate}
  The first case cannot occur because of our hypothesis \ref{RE}. The second and third cases are ruled out as explained in \cite[Proposition 4.36]{Cazassus}. It follows that only the fourth case can occur.

  For transversality reasons, the zero-dimensional component has to be compact, as in the proof of theorem \ref{TransversalityI}. Analogously to the proof of that theorem, the one-dimensional component can only lose compactness if one of the gradient line shrinks to zero, a strip (quilted or non quilted) breaks at an incoming or outgoing end, or gradient lines break. However, as explained in that proof, the shrinking of the gradient lines and the breaking at an incoming or outgoing end come in pairs, so we glue the compactifications together along these common ends. Finally, we can only have gradient line breaking, which amounts to the described boundary.
\end{proof}
We refer to elements of the boundary of the 1-dimensional compactification as \emph{broken quilt-cascades}.

We would be tempted to now define a pushforward between the Floer complexes by counting the zero-dimensional component of the moduli space of quilt-cascades. There is a problem: a quilted strip can have negative symplectic area, and so would yield an element in $\Lambda$ and not $\Lambda_0$ (which we need for the proof of the main theorem). We work around this issue using the following proposition:
\begin{prop}
  There exists an open neighborhood $U$ of $0$ in the space $\mathcal{H}$ of Hamiltonian perturbations such that for any $H\in U$, all quilted strips $\underline{u}\in \mathcal{Q}(\mathcal{I}_n,H)$ satisfy $\omega(\underline{u})\geq 0$ and
  $$\omega(\underline{u})=0 \Leftrightarrow [\underline{u}]=0$$
  where $[\underline{u}]$ is the homotopy class of $\underline{u}$.
\end{prop}
\begin{proof}
  We use Gromov compactness. If the proposition were false, there would exist a sequence of Hamiltonian perturbations $H_{\nu}\to 0$ and a sequence of quilted $H_{\nu}$-perturbed strips $\underline{u}_{\nu}$ that satisfy $\omega(\underline{u}_{\nu})<0$ but $[\underline{u}_{\nu}]\neq 0$. By passing to a subsequence, we may assume that $\underline{u}_{\nu}$ converges to some $\underline{u}$. However, this unperturbed quilted strip satisfies $\omega([\underline{u}])=E(\underline{u})=0$, hence is constant. But since $\underline{u}_{\nu}$ converges to $\underline{u}$, at some point the homotopy classes must be trivial, which is a contradiction.
\end{proof}
Now, since the set of regular Hamiltonian perturbations is of the second Baire category, then it is of the second Baire category when intersected with $U$. Since any open subspace of a completely metrizable space is completely metrizable, this subspace is also Baire. We will immediately shrink $U$ again to satisfy the following lemma:
\begin{lema}
  There exists an open neighborhood $V$ of $0$ in the space of Hamiltonian perturbations such that for all $H\in V$ and any critical points $p\in \crit f_n, q\in \crit f_{n+1}$ with $\ind_{f_n}(p)-\ind_{f_{n+1}}(q)=0$, the moduli space of broken quilt-cascades from $p$ to $q$ that represent the trivial homotopy class is empty.
\end{lema}
\begin{proof}
  If the proposition was false we would have a sequence of broken quilt-cascades converging to a broken grafted Morse line, however by transversality reasons (and a dimension count) this space is empty.
\end{proof}

We will restrict our domain of Hamiltonian perturbations to the connected component of $0$ in $U\cap V$, which we denote by $W$. Again, this is harmless as the space of regular Hamiltonian perturbations is still of the second Baire category in this space. We denote the space of perturbations in this open subset as $\operatorname{Ham}_0(X_n,X_{n+1})$.

Define the \emph{increment maps} as the $\Lambda_0$-linear extension of the maps:
\begin{eqnarray*}
  \alpha_n: FC_{\ast}(L_0^{n},L_1^{n};\Lambda_0)&\to & FC_{\ast}(L_0^{n+1},L_1^{n+1})\\
  p & \mapsto & \sum \#_2 \mathcal{QC}_{[0]}(p,q,A) T^{\omega(A)}\cdot q.
\end{eqnarray*}
The compactification of the 1-dimensional space $\mathcal{QC}_{[1]}(p,q,A)$ shows that these are in fact chain maps.
\begin{prop}\label{ZeroEnergyPart}
  The zero energy part of the increment map coincides with the Morse pushforward
  $$\Id\times_G \iota_{n_{\ast}}\otimes \Id: MC_{\ast}(L_0^{n}\cap L_1^{n})\otimes \Lambda_0\to MC_{\ast}(L_0^{n+1}\cap L_1^{n+1})\otimes \Lambda_0.$$
\end{prop}
\begin{proof}
  We will use a cobordism argument analogous to the one in \cite[Proposition 5.3.16]{Biran-Cornea}. By the previous proposition, we know that the zero energy part is precisely when $A=0$. Now starting from the regular $H\in W$, pick a generic path $H_{t}$ in $\operatorname{Ham}_0(X_n,X_{n+1})$ such that $H_{0}=H$ and $H_{1}=0$. Given $p,q$ such that $\ind_{f_n}(p)-\ind_{f_{n+1}}(q)=0$, form the moduli space

  $$\mathcal{C}= \{(Q,t)\mid Q\in \mathcal{QC}_{[0]}(p,q,0;\mathcal{I}_n, H_{t}), t\in [0,1]\}.$$

  Note that since $A=0$, we cannot have any non-quilted jumps, as they would have zero energy and thus be constant (remember that the non-quilted jumps are \emph{unperturbed} so their energy and area coincide). Furthermore, $\mathcal{QC}_{[0]}(p,q,0;\mathcal{I}_n, 0)$ coincides with the moduli space of grafted Morse lines defined in section \ref{PushforwardSection}. The space $\mathcal{C}$ is a smooth manifold whose boundary is given by
  \begin{align*}
    \partial \mathcal{C}&= \mathcal{QC}_{[0]}(p,q,0;\mathcal{I}_n, H)\bigsqcup \mathcal{QC}_{[0]}(p,q,0;\mathcal{I}_n, 0)\\
                        &= \mathcal{QC}_{[0]}(p,q,0;\mathcal{I}_n, H)\bigsqcup\mathcal{MG}_{\Id\times_{G}\iota_n}(p,q).
  \end{align*}
  We claim that $\mathcal{C}$ is compact. Indeed, the only possible loss of compactness would come from gradient line breaking (we can't have strip-breaking because the homotopy class is trivial, so the resulting strips would be constant), but by the previous lemma and our choice of $\operatorname{Ham}_0(X_n,X_{n+1})$ this is impossible. Thus:
  $$\#_2\mathcal{QC}_{[0]}(p,q,0;\mathcal{I}_n, H)= \#_2\mathcal{MG}_{\Id\times_{G}\iota_n}(p,q)$$
  and the proposition is proved.
\end{proof}
Note that we only need to choose the neighborhood $\operatorname{Ham}_0(X_n,X_{n+1})$ for a \emph{given} pair of Floer complexes, so we don't need any sort of `universal choice'.

We define the \emph{equivariant Floer chain complex}, denoted $FC^{G}_{\ast}(L_0,L_1;\Lambda_0)$ as the telescope of the $FC_{\ast}(L_0^{n},L_1^{n};\Lambda_0)$ with the maps $\alpha_n$. In symbols
$$FC_{\ast}^{G}(L_0,L_1;\Lambda_0):= \operatorname{Tel}\left( FC_{\ast}(L_0^{n},L_1^{n};\Lambda_0),\alpha_n \right).$$
We define the \emph{equivariant Floer homology} of $L_0$ and $L_1$ over $\Lambda_0$ as the homology of this chain complex, and we denote it by $FH^{G}_{\ast}(L_0,L_1;\Lambda_0)$.
\begin{obs} We immediately observe the following:\begin{itemize}
  \item The zero energy part of $FC_{\ast}^{G}(L_0,L_1;\Lambda_0)$ coincides the equivariant Morse complex of $L_0\cap L_1$.
  \item The resulting homology is independent of auxiliary choices of: universal Morse data, choice of almost complex structures and choice of Hamiltonian perturbations (as long as they are small enough).
  \item $FC_{\ast}(L,L;\Lambda_0)\cong MC^{G}_{\ast}(L)\hat{\otimes}_{\Z_2}\Lambda_0.$
  \end{itemize}.
\end{obs}
\subsection{Gapping}\label{Gapping}
Recall that we denote by $\operatorname{Ham}_0 (X_n, X_{n+1})$ the space of Hamiltonian perturbations that are `small enough' (in the sense of the previous discussion). We aim to show in this section that for a suitable choice of perturbation data, the equivariant Floer complex is gapped. By definition of the telescope differential, we need to find data such that each Floer complex $FC_{\ast}(L_0^{n},L_1^{n};\Lambda_0)$ and each increment morphism $\alpha_n$ is $\hbar$-gapped where $\hbar$ is independent of $n$.

We start by showing that the differentials are uniformly gapped. Fix a $G$-invariant almost complex structure $J_{G}$ on $X$. According to proposition \ref{Gap1}, there exists a constant $\hbar>0$ such that any non-constant $J_{G}$-holomorphic strip $u$ on $X$ with boundaries on $L_0$ and $L_1$ satisfies $E(u)=\omega(u)\geq \hbar >0$. Now fixing $J_{G}\oplus \pi_2^{\ast}J_{T^{\ast}BG_n}$, any non-constant holomorphic strip $u'$ necessarily satisfies $E(u')=\omega(u')\geq \hbar$. Indeed, any such $u'$ is contained in some fiber, symplectomorphic to $X$ and with almost complex structure $J_G$. By virtue of proposition \ref{Gap2}, given $\varepsilon$ with $0<\varepsilon<\hbar$, there exists an open neighbourhood of this reference almost complex structure such that any non-constant pseudoholomorphic strip has energy at least $\hbar-\varepsilon$. We denote this open neighbourhood as $U_{\varepsilon,n}\subset \mathcal{C}^{\infty}([0,1],\mathcal{J}_{\operatorname{adm}}^{\operatorname{split}}(X_n))$. Thus for any such $\varepsilon$, we can find generic almost complex structures such that all the Floer complexes are $\hbar-\varepsilon$-gapped.

We now need to show that the area of non-constant quilted strips can be uniformly bounded away from $0$. Note that we can view $J_G\oplus\pi_2^{\ast}J_{T^{\ast}BG_n}$ and $J_G\oplus \pi_2^{\ast}J_{T^{\ast}BG_n}$ as a constant \emph{quilted} datum on the strip. Consider a quilted strip with respect to this datum, and denote it by $\underline{u}=(u_1,u_2)$. Then $\pi_2\oplus\pi_2\circ \underline{u}:Z\to T^{\ast}BG_n\times T^{\ast}BG_{n+1}$ is homotopically trivial, holomorphic and hence constant. Then $\underline{u}$ can be seen as a quilted map to $X$, but note that the Lagrangian correspondence $\mathcal{I}_n$ restricted to a fiber is just the diagonal: hence $\underline{u}:Z\to X$ is just a $J_G$-holomorphic strip, thus has minimal energy $\hbar>0$. By Gromov compactness, there exists a neighbourhood $V$ of this fixed quilted almost complex structure such that any quilt has minimal energy $\hbar-\varepsilon>0$. The map that to a quilted pair associates its asymptotic values, $\underline{J}\mapsto (J_{-\infty},J_{+\infty})\subset \mathcal{J}^{\operatorname{split}}_{\operatorname{adm}}(X_n)\times \mathcal{J}^{\operatorname{split}}_{\operatorname{adm}}(X_{n+1})$ is open, and thus the image of $V$ is an open subset of the product. Up to shrinking this subset, we may assume it is of the form $V^{n,n+1}_{\varepsilon,n}\times V^{n,(n+1)}_{\varepsilon,n+1}$ (we add the superscript because this construction depends on a given pair) Set $W_{\varepsilon,n}=U_{\varepsilon,n}\cap V^{(n-1,n)}_{\varepsilon,n}\cap V^{(n,n+1)}_{\varepsilon,n}$ (i.e. we intersect $U_{\varepsilon,n}$ with the above neighbourhoods coming from the previous space and the next). Now take the product $\prod_{n\in \mathbb{N}} W_{\varepsilon,n}$, being a countable product of open subsets of completely metrizable spaces it is completely metrizable and hence Baire, so we can find regular almost complex structures without any issue. Furthermore, for any tuple of complex structures in this product we have that each differential on $X_n$ is $\hbar-\varepsilon$-gapped and there exist (for any two consecutive indices) quilted data that are $\hbar-\varepsilon$-gapped.

Finally, we proceed similarly with the product space of all Hamiltonian perturbations: i.e. up to shrinking $\operatorname{Ham}_0(X_n,X_{n+1})$ we can ensure that there are generic Hamiltonian perturbations such that the space of non-homotopically trivial quilted strips remain $\hbar-\varepsilon$-gapped.

\subsection{Convexity I}\label{Conv1}
Something we skipped in the compactness argument is the role of convexity, namely we need to show that in the $G$-convex setting the quilt-cascades don't escape to infinity. Note that by the maximum principle, regular holomorphic strips don't escape to infinity, and similarly given a quilted strip $\underline{u}$, $\tilde{\rho}_n\oplus \tilde{\rho}_{n+1}\circ \underline{u}$ cannot achieve its maximum at a point in the interior of a patch. It could achieve its maximum at the true boundary components, but since we assume them to be compact this poses no problem. It remains to rule out the case where the maximum is achieved at the seam.

Suppose the maximum was achieved at the seam and denote by $(0,t_0)$ this point. By folding and using the strong maximum principle (when we fold the seam becomes a boundary component) we get that, at $\left(0,t_0\right)$:
$$d\tilde{\rho}_{n}\left(\frac{\partial}{\partial s}\right)-d\tilde{\rho}_{n+1}\left(\frac{\partial }{\partial s}\right)>0.$$
However, since $\cfrac{\partial\underline{u}}{\partial t}_{(0,t_0)}$ is tangent to the seam, it splits (with respect to the symplectic connection) as $(\xi,\eta)\oplus(\xi, \eta')$ i.e. a `diagonal part' on the fiber and a `conormal part' on the base. Since we are using the same structure $J_C$ on the fiber we get that $d\tilde{\rho}_n(J_C\xi)=d\tilde{\rho}_{n+1}(J_C \xi)$ and hence the difference is identically zero. Now the pair $(\eta,\eta')$ is tangent to the conormal bundle  $N_{\Gamma(\iota_n)}$ which is actually Legendrian for the contact structure. Similarly to the proof of theorem 2.1 of \cite{Oh3}, the difference $d\tilde{\rho}_n(J_{C}\eta)-d\tilde{\rho}_{n+1}(J_C\eta')$ must be zero.
\section{Proof of the main theorem}

\subsection{Hypotheses} In this subsection we assume that the action of $G$ on $\mu_{X}^{-1}(0)$ is free. We also assume for simplicity that the reduction $X\slash\slash G$ satisfies hypothesis \eqref{TH}, this is true for example as soon as $G$ is connected. Again using the techniques from appendix \ref{Appendix} this hypothesis can be removed.

\subsection{The projection maps}
Motivated by the Marsden-Weinstein correspondence we define the following Lagrangian correspondence on $X_n\times \overline{X\slash \slash G}$ as the image of the maps:
$$\Pi_n:= i\times \pi_1: \mu_{X}^{-1}(0)\times_G 0_{EG_n}\rightrightarrows X_n \times \overline{X\slash\slash G}$$
where $i$ is the inclusion and $\pi_1$ is the induced map from the composition of the first coordinate projection $\mu_{X}^{-1}(0)\times 0_{EG_n}\to \mu_{X}^{-1}(0)$ with the canonical projection $\mu_{X}^{-1}(0)\to X\slash\slash G$. This is easily checked to be a Lagrangian correspondence.

In similar fashion to the increment maps, we decorate the quilted surface $\underline{Z}$ as follows: we assign the manifold $X_n$ to $P_0$, the manifold $X\slash\slash G$ to $P_1$, the Lagrangian correspondence $\Pi_n$ to the seam $S$, and the Lagrangian submanifolds $L^{n}_j$ to $\partial^{j}P_0$ and $\overline{L_j}$ to $\partial^{j}P_1$, as in the following picture:
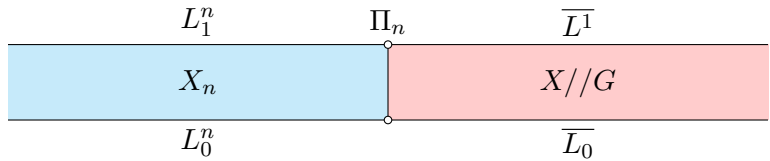
\begin{figure}[h]
  \centering
    \begin{tikzpicture}
    \fill[cyan!20!white] (0,0) rectangle +(5,1);
    \fill[red!20!white] (5,0) rectangle +(5,1);
    \draw (0,1) -- (10,1);
    \draw (0,0) -- (10,0);
    \draw (5,0) -- (5,1)node[above]{$\Pi_{n}$};
    \path (2.5,0.5)node{$X_{n}$}--(7.5,0.5)node{$X//G$};
    \path (2.5,-0.3)node{$L^{n}_{0}$}--+(5,0)node{$\overline{L_{0}}$};
    \path (2.5,1.3)node{$L^{n}_{1}$}--+(5,0)node{$\overline{L^{1}}$};
    \path[every node/.style={draw,circle,fill=white,inner sep=1pt}] (5,1)node{}--(5,0)node{};
  \end{tikzpicture}
  \caption{Decoration of $\underline{Z}$.}
  \label{fig:decoration-underline-Z}
\end{figure}

For a given quilted datum $(\underline{J},\underline{H})$ and two connected components $C_{-}\subset L_0^{n}\cap L_1^{n}$, $C_{+}\subset \overline{L_0}\cap \overline{L_1}$ we denote by $\mathcal{Q}(\Pi_n, C_{-},C_{+})$ the space of pseudoholomorphic quilts $\underline{u}=(u_1,u_2)$ (with respect to the stated decoration) such that $u_1(-\infty)\in C_{-}$ and $u_2(+\infty)\in C_{+}$.

The following theorem is proven exactly the same way as theorem \ref{TransversalityII}.
\begin{theorem}
  For a given quilted almost complex structure $\underline{J}$, there exists a comeagre subset in the space of Hamiltonian perturbations such that for any two connected components $C_{-}$ and $C_{+}$, the moduli space $\mathcal{Q}(\Pi_n,C_{-},C_{+})$ is transversely cut-out and the evaluation maps at $\pm\infty$ are submersive.
\end{theorem}
We also will need the following proposition:
\begin{prop}
  If $\underline{H}=0$, the space of constant $\Pi_n$-quilts is transversely cut-out and diffeomorphic to the graph
  $$\Gamma(\pi_1)\subset (L_0^{n}\cap L_1^{n})\times (\overline{L_0}\cap \overline{L_1})$$
  where $\pi_1: L_0^{n}\cap L_1^{n}\to \overline{L_0}\cap \overline{L_1}$ is the quotient of the first-coordinate projection.
\end{prop}
\begin{proof}
  We sketch the proof. The space of constant quilts is clearly diffeomorphic to the given graph, by definition. Thus it is sufficient to show that this space is transversely cut-out for a \emph{given} quilted almost complex structure $\underline{J}$. Indeed, since the index is independent of the choice of $\underline{J}$, surjectivity with respect to a given quilted datum implies surjectivity for all as the dimension of the kernel will remain unchanged.

  As expected, we fix a $G-$invariant almost complex structure $J_G$ on $X$, which induces an almost complex structure $\overline{J_G}$ on $X\slash\slash G$, and define the quilted datum by $J_0\equiv J_{G}\oplus \pi_2^{\ast}J_{T^{\ast}BG_n}$ and $J_1\equiv \overline{J_G}$. Given a constant quilt $\underline{u}=(u_1,u_2)\in \Pi_n$, we trivialize the tangent space at $u_1$ (which is just a point) as
  $$u_1^{\ast}TX_{n}\cong T_{u_1'}X\oplus \pi_{2}^{\ast}T_{\pi_2(u_1)} T^{\ast}BG_n$$
  where $u_1'\in \mu_X^{-1}(0)$ is the point in the fiber, identified to $X$. Note that with respect to the second summand surjectivity is assured: indeed it reduces to the case of constant holomorphic maps from the patch $P_0$ to $T^{\ast}BG_n$ with boundary in $0_{BG_n}$, which by index reasons is transversely cut-out. We may thus reduce the problem to showing that the linearized Cauchy-Riemann operator at constant pairs $(u_1', u_2)$ where $u_1'\in \mu_{X}^{-1}(0), u_2=\pi(u_1)\in X\slash\slash G$ is surjective. By our choice of $J_G$, we have a splitting
  $$T_{u_1'}X\cong T_{u_1'}\mathcal{O}_{u_1'}\oplus \left(T_{u_1'}\mathcal{O}_{u_1'}\right)^{\bot}\oplus J_G\cdot T_{u_1'}\mathcal{O}_{u_1'}$$
  where $\mathcal{O}_{u_1'}$ is the orbit of $G$ going through $u_1'$, and the term $(T_{u_1'}\mathcal{O}_{u_1'})^{\bot}$ denotes the orthogonal complement (with respect to the metric induced by $J_G$) of $T_{u_1'}\mathcal{O}_{u_1'}$ in the space $T_{u_1}\mu_{X}^{-1}(0)$. There is a $J_{G}$ holomorphic identification of $T_{u_2}X\slash\slash G$ with $(T_{u_1'}\mathcal{O}_{u_1'})^{\bot}$ by definition of the induced almost complex structure.

  All the mentioned splittings are holomorphic, so the Cauchy-Riemann operator splits into a direct sum of Cauchy-Riemann operators taking values in a direct sum. We can split the problem in two parts
  \begin{itemize}
  \item Surjectivity of Cauchy-Riemann operator from the space of maps from $P_0$ to $T_{u_1'}\mathcal{O}_{u_1'}\oplus J_{G}\cdot T_{u_1'}\mathcal{O}_{u_1'}$ with boundary conditions $T_{u_1'}\mathcal{O}_{u_1'}$ (note these are Lagrangian). But as mentioned before, the Cauchy-Riemann operator is surjective on maps with the same Lagrangian as boundary conditions.
  \item Surjectivity of the restriction to quilted sections decorated by the diagonal and with boundaries in $T_{u_1'}L_j\cap (T_{u_1'}\mathcal{O}_{u_1'})^{\bot}$, but this problem reduces to the surjectivity of the operator for constant strips, where we know it holds (this can once again, this can be verified using \cite{Schmaschke}'s dimension formula or directly as in \cite{Lekili-Lipyanskiy}, section 2.4).
  \end{itemize}
  Thus the linearized Cauchy-Riemann operator is surjective on each part of the splitting, so surjectivity follows.
\end{proof}
Compactness is a bit more subtle, and it will follow from the following analysis on $\Pi_n$-quilted strips. A sequence of such strips can degenerate to one of the following cases:
\begin{enumerate}
\item Develop a bubble in a true boundary or the interior of a patch,
\item quilted strip breaking at the free end,
\item quilted bubbling at the seam,
\item strip breaking at an incoming or outgoing end.
\end{enumerate}
We will show that the only case that can happen is the last one.

The first case cannot happen because of hypothesis \ref{RE} on both $X_n$ and $X//G$.

We draw a picture of the second case:
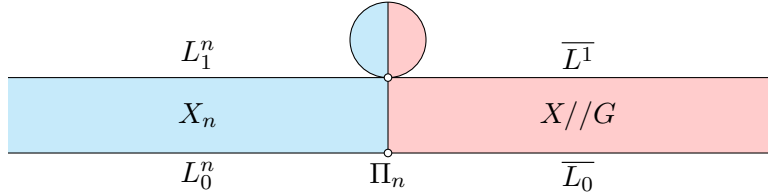
\begin{figure}[h]
  \centering
    \begin{tikzpicture}
    \fill[cyan!20!white] (0,0) rectangle +(5,1) (5,1)--+(0,1)arc(90:270:0.5);
    \fill[red!20!white] (5,0) rectangle +(5,1) (5,1)--+(0,1)arc(90:-90:0.5);
    \draw (0,1) -- (10,1);
    \draw (0,0) -- (10,0);
    \draw (5,0) node[below]{$\Pi_{n}$}-- (5,2);
    \path (2.5,0.5)node{$X_{n}$}--(7.5,0.5)node{$X//G$};
    \path (2.5,-0.3)node{$L^{n}_{0}$}--+(5,0)node{$\overline{L_{0}}$};
    \path (2.5,1.3)node{$L^{n}_{1}$}--+(5,0)node{$\overline{L^{1}}$};
    \draw (5,1.5) circle (0.5);
    \path[every node/.style={draw,circle,fill=white,inner sep=1pt}] (5,1)node{}--(5,0)node{};
  \end{tikzpicture}
  \caption{Strip breaking at one of the free ends}
  \label{fig:strip-break-free-end}
\end{figure}
We claim we can lift such a strip to a disc on $X_n$ with boundary in $L^{n}_1$. Indeed, if we write it as $\underline{u}=(u_1,u_2)$, then since $\mu_{X}^{-1}(0)\times_G 0_{EG_n}\to X\slash\slash G$ is a fiber bundle we can lift $u_2$ (since its domain is contractible) and we can choose such a lift such that it agrees at the seam with $u_1$.

Now this disc has zero symplectic area, and thus is constant.

In the third case, we use a trick from \cite[Proposition 4.36]{Cazassus}: we can follow the seam to a point in $L_1^{n}\times \overline{L_1}$ and see it as a quilted disc as in case 2: it must then have zero area and thus be constant. This argument is similar to the way we deduce that if $L\subset X$ is relatively exact, then $X$ must be aspherical.

Finally, the only case that can actually happen is strip breaking at an incoming or outgoing end.

We play the same game as before: we fix a Morse function $\overline{f}:\overline{L_0}\cap \overline{L_1}\to \R$ together with a Morse-Smale metric and we define the space of $\Pi_n$-quilt cascades as before and denote the moduli space as $\mathcal{QC}(p,q,A;\Pi_n)$. This serves to define a map
\begin{eqnarray*}
  \mathfrak{P}_n: FC_{\ast}(L_0^{n},L_1^{n};\Lambda_0)&\to & FC_{\ast}(\overline{L_0},\overline{L_1};\Lambda_0)\\
  p & \mapsto & \sum_{q} \#_2 \mathcal{QC}_{[0]}(p,q,A;\Pi_n) T^{\omega(A)}\cdot q.
\end{eqnarray*}
The compactification of the 1-dimensional moduli space implies this is a chain map.

\textbf{Warning:} As before, we need to restrict ourselves to a smaller class of perturbation data (i.e. sufficiently small Hamiltonian perturbations) to ensure that the element on the right is in $\Lambda_0$, in other words such that for any quilted strip $\underline{u}$ we have $\omega(\underline{u})=0$ if and only if $[\underline{u}]=0$.

With this caveat in mind, we can ask what the zero-energy part of $\mathfrak{P}_n$ is. We will require the following technical lemma:
\begin{lema}
  There exists universal Morse-Smale data $((f_0,g_0),(f_1,g_1),\dots)$ on all the $L_0^{n}\cap L_1^{n}$ and $(\overline{f},\overline{g})$ on $\overline{L_0}\cap \overline{L_1}$ such that the space of grafted Morse lines $\mathcal{MG}_{\pi_1^{n}}(p,q)$ is transversely cut out for all $\pi_1^{n}: L_0^{n}\cap L_1^{n}\to \overline{L_0}\cap \overline{L_1}$, where $\pi_1^{n}$ is the quotient of the projection onto the first coordinate.
\end{lema}
\begin{proof}
  The proof is a genericity statement similar as to those done before and is omitted.
\end{proof}

Finally, as with proposition \ref{ZeroEnergyPart} we can show that the zero energy part coincides (after choosing some possibly smaller open subset of Hamiltonian perturbations) with the pushforward
$$\pi_{1_{\ast}}^{n}\otimes \Id: MC_{\ast}(L_0^{n}\cap L_1^{n})\otimes \Lambda_0 \to MC_{\ast}(\overline{L_0}\cap\overline{L_1})\otimes \Lambda_0.$$
\subsection{Convexity II}
In the case where $X$ is non-compact we have to once again justify that quilted discs do not escape to infinity. Note that hypothesis \ref{GC} combined with lemma \ref{LemmaConvex} give a convex structure $\overline{\rho}$ on $X\slash\slash G$. Similarly to section \ref{Conv1} the only problem is if a maximum is reached at the seam. After folding and identifying the quilt as a map to the product, exists a point $(0,t_0)$ such that $(d\tilde{\rho}_n-d\overline{\rho})\cfrac{\partial \underline{u}}{\partial s}>0$. By construction of $\tilde{\rho}_n$ and since points of $\Pi_n$ lie in the zero section, we conclude that $\rho_{EG_n}$ is constant and we are reduced to $(d\rho-d\overline{\rho})(\partial_s \underline{u})>0$. If we denote by $u_1$ and $u_2$ the components of $\underline{u}$ we have $d\rho(\partial_s u_1)-d\overline{\rho}(-\partial_s u_2)>0$. However, we have that the pair $(\partial_t u_1,\partial_t u_2)$ is tangent to the seam so one the second vector is the quotient projection of the first. Split the tangent space at $x=u_1(0,t_0)$ into $T_{x}\mathcal{O}^{\bot}\oplus T_{x}\mathcal{O}\oplus J_G T_{x}\mathcal{O}$ where $\mathcal{O}$ denotes the $G$-orbit at $x$ and accordingly split $\partial_s u_1=(\xi,\eta, \eta')$. The fact that $\partial_t u_2$ is the quotient of $\partial_t u_1$ implies that $\partial_s u_2(-\partial_s)=\xi$ and by construction of $\overline{\rho}$ we get that $d\rho(\xi)-d\overline{\rho}(-\partial_s u_2)=0$. It remains to show that $d\rho(\eta,\eta')=0$, but this follows from $G$-invariance of $\rho$ and the fact that $d\rho(JX_{\xi})=0$ on $\mu^{-1}(0)$. We thus get a contradiction. 
This shows that $\Pi_n$-quilted discs remain in a compact set.
\subsection{The projection maps commute up to homotopy}

We start with the following proposition:
\begin{prop}
  The composition $\mathcal{I}_n \circ \Pi_{n+1}$ is embedded and equal to $\Pi_{n}$.
\end{prop}
\begin{proof}
  Pick a point $([x, (q,p)], z)\in X_n\times \overline{X\slash\slash G}$ that lies in the composition. By definition, this means that there exists a point $[x',(q',p')]\in X_{n+1}$ such that
  $$\begin{cases}
    ([x,(q,p)],[x',(q',p')])\in \mathcal{I}_n,\\
    ([x',(q',p')],z)\in \Pi_{n+1}.
  \end{cases}$$
  From the second line we get that $p'=0$ and $\pi(x')=z$. Plugging this into the first line we get that $p=0$, that $x\in \mu_{X}^{-1}(0)$ and that $\pi(x)=z$. From this it follows that $([x,(q,0)],z)\in \Pi_n$.

  Now we need to show that the element $[x',(q',p')]$ is unique (this shows that the canonical projection is injective). But this is clear because the only element that can satisfy both conditions is $[x',(q',0)]= \Id\times_G \iota_n [x,(q,0)]$. This completes the proof.
\end{proof}
We will construct a homotopy between the different spaces of quilt cascades involved. Fix a real parameter $R\in [0,\infty)$ and fix some small $\epsilon>0$. For $R\in [\epsilon,\infty)$ we define a quilted surface $\underline{Z}^{R}$ decorated as in the following picture:
\begin{figure}[h]
  \centering
  \begin{tikzpicture}
    \fill[cyan!20!white] (0,0) rectangle +(4,1);
    \fill[yellow!20!white] (4,0) rectangle +(2,1);
    \fill[red!20!white] (6,0) rectangle +(4,1);
    
    \draw (0,0)--(10,0) (0,1)--+(10,0) (4,0)--+(0,1) (6,0)--+(0,1);
    \path (0,0.5) +(2,0) node{$X_{n}$} +(5,0) node{$X_{n+1}$} +(8,0) node{$X//G$};
    \path (0,1) +(2,0) node[above]{$L_{1}^{n}$} +(5,0) node[above]{$L_{1}^{n+1}$} +(8,0) node[above]{$\overline{L_{1}}$} +(4,0) node[above]{$\mathcal{I}_{n}$} +(6,0)node[above]{$\Pi_{n+1}$};
    \path (0,0) +(2,0) node[below]{$L_{0}^{n}$} +(5,0) node[below]{$L_{0}^{n+1}$} +(8,0) node[below]{$\overline{L_{0}}$} +(4,0) node[below]{$0$} +(6,0)node[below]{$R$};
  \end{tikzpicture}
  \caption{Quilted surface $\underline{Z}$. The two seams are located at positions $0$ and $R$, and are labelled by $\mathcal{I}_{n}$ and $\Pi_{n+1}$.}
  \label{fig:quilted-surface-Z}
\end{figure}
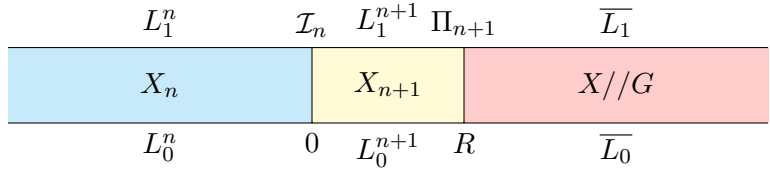

A naïve approach would be to let $R\to 0$: as the seams converge they collapse to the seam decorated by the composition, provided one can rule out figure-eight bubbling (cf. section 5 of \cite{Wehrheim-Woodward1}). But as pointed out \cite[page 41]{Cazassus}, this cannot work as two strip-like ends (namely the free ends) come together, so a bit more care is needed and we follow his approach. The idea is to ``stretch off'' the strip-like ends to obtain another surface with only two quilted free ends, then the strip shrinking analysis from \cite{Wehrheim-Woodward1} applies (again, we have to be careful about figure-eight bubbles, cf, \cite{Bottman}).

Between $\epsilon$ and $\epsilon'\in (0,\epsilon)$ we deform the moduli space as in the following picture:

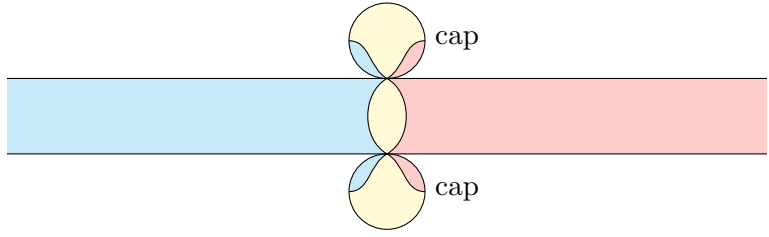
\begin{figure}[h]
  \centering
  \begin{tikzpicture}
    \fill[yellow!20!white] (5,0) to[out=30,in=-30] (5,1) to[out=-150,in=150] (5,0) (5,1.5) circle (0.5) (5,-0.5) circle (0.5);
    \fill[cyan!20!white] (0,0) -- (5,0) to[out=150,in=-150] (5,1)--(0,1)--cycle (4.5,-0.5) to[out=0,in=-150] (5,0) arc (90:180:0.5) (5,1) to[out=150,in=0] (4.5,1.5) arc (180:270:0.5);

    \fill[red!20!white] (10,0) -- (5,0) to[out=30,in=-30] (5,1)--(10,1)--cycle (5.5,-0.5) to[out=180,in=-30] (5,0) arc (90:0:0.5) (5,1) to[out=30,in=180] (5.5,1.5) arc (0:-90:0.5);

    \draw (0,0)--(10,0) (0,1)--+(10,0) (4.5,-0.5) to[out=0,in=-150] (5,0) to[out=30,in=-30] (5,1) to[out=150,in=0] (4.5,1.5) (5.5,-0.5) to[out=180,in=-30] (5,0) to[out=150,in=-150] (5,1) to[out=30,in=180] (5.5,1.5);
    \draw (5,1.5) circle (0.5) (5,-0.5) circle (0.5);

    \path (5.5,1.5) node[right] {cap} (5.5,-0.5) node[right]{cap};
  \end{tikzpicture}
  \caption{By parameter value $\epsilon'$ we bring together the two strip-like ends.}
  \label{fig:quilted-surface-pinch}
\end{figure}

As shown in Figure \ref{fig:quilted-surface-pinch}, this process yields two quilted strips (seen as ``caps'' on the picture), however by hypothesis \ref{RE} these are constant.

Finally, for $R\in [0,\epsilon')$ we make the seams collapse to then obtain a quilted strip with seam decorated by the composition of the correspondences, which by the previous proposition is just $\Pi_n$.

During this step, as the seams collide a figure-eight bubble can be formed, see \cite{Bottman} for a removal of singularity theorem. However, by using the trick we used before, namely taking a path that connects to a Lagrangian product and ``opening up'' the figure-eight bubble we obtain a quilted strip which must have zero area and thus be constant.

We denote by $\mathcal{K}_{R}$ the moduli space of pseudoholomorphic quilts associated to this moduli space. Note that as $R\to \infty$ we obtain two quilted strips (each with the corresponding Lagrangian correspondence) 'kissing' at the endpoints.

Now given two critical points $p\in \crit f_n, q\in \crit \overline{f}$, we define the space of $\mathfrak{H}$-quilt cascades from $p$ to $q$ representing the homotopy $A$ as one of the following two cases:
\begin{enumerate}
\item Either it consists of a cascade with two quilted jumps, one for each Lagrangian correspondence,
\item or it is a quilt-cascade with one quilted jump decorated by $\mathcal{K}_{R}$ for some $R$.
\end{enumerate}
We illustrate the two cases:
\begin{figure}[h]
  \centering
  \begin{tikzpicture}
    \path[decorate,decoration={
      markings,
      mark=between positions 0.65 and 1 step 0.5 with {\arrow{>[scale=1.3,line width=.8pt]};},
    }] (0.5,0) to (1,0);
    \path[decorate,decoration={
      markings,
      mark=between positions 0.6 and 1 step 0.5 with {\arrow{>[scale=1.3,line width=.8pt]};},
    }] (3,0) to (4,0);
    \path[decorate,decoration={
      markings,
      mark=between positions 0.6 and 1 step 0.5 with {\arrow{>[scale=1.3,line width=.8pt]};},
    }] (6,0) to (7,0);
    \path[decorate,decoration={
      markings,
      mark=between positions 0.6 and 1 step 0.5 with {\arrow{>[scale=1.3,line width=.8pt]};},
    }] (9,0) to (10,0);
    \path[decorate,decoration={
      markings,
      mark=between positions 0.65 and 1 step 0.5 with {\arrow{>[scale=1.3,line width=.8pt]};},
    }] (12,0) to (12.5,0);
    
    \draw (0.5,0) node[below] {$p$} node[fill, inner sep=1pt, circle] {}-- (12.5,0)node[fill, inner sep=1pt, circle] {}node[below]{$q$};
    \path[fill=cyan!20!white] (2,0) circle (1 and 0.7) (5,0.7) arc (90:270:1 and 0.7)--cycle;
    \path[fill=red!20!white] (11,0) circle (1 and 0.7) (8,0.7) arc (90:-90:1 and 0.7)--cycle;
    \path[fill=yellow!20!white] (11,0) (8,0.7) arc (90:270:1 and 0.7)--cycle (5,0.7) arc (90:-90:1 and 0.7)--cycle;
    
    \draw (5,0.7)--(5,-0.7) (8,0.7)--(8,-0.7);
    \draw (2,0) circle (1 and 0.7) node{$X_{n}$} (5,0) circle (1 and 0.7) (8,0) circle (1 and 0.7) (11,0) circle (1 and 0.7) node{$X//G$};
  \end{tikzpicture}
  \caption{A quilt cascade with two quilted jumps, one decorated by each Lagrangian correspondence}
  \label{fig:quilt-cascades-two-quilted-jumps}
\end{figure}
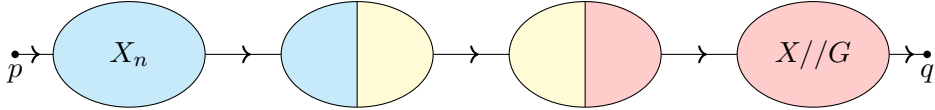

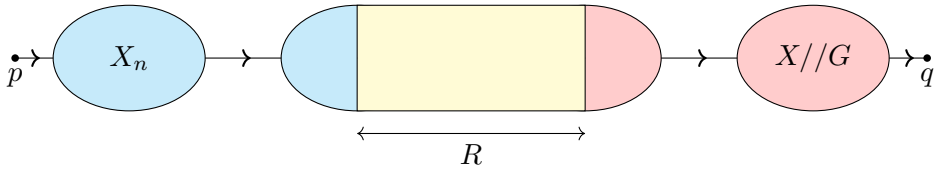
\begin{figure}[h]
  \centering
  \begin{tikzpicture}
    \path[decorate,decoration={
      markings,
      mark=between positions 0.65 and 1 step 0.5 with {\arrow{>[scale=1.3,line width=.8pt]};},
    }] (0.5,0) to (1,0);
    \path[decorate,decoration={
      markings,
      mark=between positions 0.6 and 1 step 0.5 with {\arrow{>[scale=1.3,line width=.8pt]};},
    }] (3,0) to (4,0);
    \path[decorate,decoration={
      markings,
      mark=between positions 0.6 and 1 step 0.5 with {\arrow{>[scale=1.3,line width=.8pt]};},
    }] (6,0) to (7,0);
    \path[decorate,decoration={
      markings,
      mark=between positions 0.6 and 1 step 0.5 with {\arrow{>[scale=1.3,line width=.8pt]};},
    }] (9,0) to (10,0);
    \path[decorate,decoration={
      markings,
      mark=between positions 0.65 and 1 step 0.5 with {\arrow{>[scale=1.3,line width=.8pt]};},
    }] (12,0) to (12.5,0);
    
    \draw (0.5,0) node[below] {$p$} node[fill, inner sep=1pt, circle] {}-- (12.5,0)node[fill, inner sep=1pt, circle] {}node[below]{$q$};
    \path[fill=cyan!20!white] (2,0) circle (1 and 0.7) (5,0.7) arc (90:270:1 and 0.7)--cycle;
    \path[fill=red!20!white] (11,0) circle (1 and 0.7) (8,0.7) arc (90:-90:1 and 0.7)--cycle;
    
    \draw (2,0) circle (1 and 0.7) node{$X_{n}$} (5,0) circle (1 and 0.7) (8,0) circle (1 and 0.7) (11,0) circle (1 and 0.7) node{$X//G$};
    \draw[fill=yellow!20!white] (5,-0.7) rectangle (8,0.7);
    \draw[<->] (5,-1)--node[below]{$R$}(8,-1);
  \end{tikzpicture}
  \caption{A quilt cascade with one quilted jump decorated by an element of $\mathcal{K}_{R}$.}
  \label{fig:quilt-cascades-one-quilted-jumps}
\end{figure}

For generic, arbitrarily small Hamiltonian perturbations (and possible change of the quilted almost complex structures, but this is irrelevant) we can achieve transversality for this moduli spaces, which we denote $\mathfrak{H}(p,q,A)$. We denote the corresponding map between the Floer complexes as $H$.

The zero-dimensional component is compact, hence a finite set. Looking at the compactification of the 1-dimensional component we can have new kinds of degeneration:
\begin{enumerate}
\item A gradient line breaks between the two quilted jumps,
\item a gradient line breaks either before or after both quilted jumps,
\item a gradient line shrinks to zero between the two quilted jumps,
\item the element of $\mathcal{K}_{R}$ in the quilted jump degenerates as $R\to \infty$,
\item the element of $\mathcal{K}_{R}$ in the quilted jump degenerates as $R\to 0$,
\item a gradient line breaks before or after the $\mathcal{K}_{R}$-jump.
\end{enumerate}
Now note that the first case corresponds to elements in $\mathfrak{P}_{n+1}\circ \alpha_n$. The second and sixth cases correspond to elements in $\partial H +H\partial$, the third and fourth cases come in pairs, so we can glue along these ends. Finally, the fifth case corresponds to elements in $\mathfrak{P}_{n}$.

This shows:
\begin{prop}
  The diagrams
  $$\begin{tikzcd}
    FC_{\ast}(L_0^{n},L_1^{n};\Lambda_0) \arrow{rr}{\alpha_n}\arrow{dr}{\mathfrak{P}_n} & & FC_{\ast}(L_0^{n+1},L_1^{n+1};\Lambda_0)\arrow{dl}{\mathfrak{P}_{n+1}}\\
    & FC_{\ast}(\overline{L_0},\overline{L_1};\Lambda_0) &
  \end{tikzcd}$$
  are homotopy commutative.
\end{prop}
Since all the maps involved respect the filtration, we get a natural map
$$\mathfrak{P}: FC_{\ast}^{G}(L_0,L_1;\Lambda_0)\to FC_{\ast}(\overline{L_0},\overline{L_1};\Lambda_0).$$
\subsection{The map $\mathfrak{P}$ is gapped}
We only sketch the proof, as it is similar to the discussion in section \ref{Gapping}. Fixing a reference $J_G$ that is $G$-invariant, the quilted strips coming from the map $\mathfrak{P}_n$ have a minimal area $\hbar'>0$ corresponding to the minimal area of non-constant quilts in $X\times X\slash\slash G$ decorated by the Weinstein correspondence, hence is independent of the space $X_n$. We can also gap the non-constant homotopies appearing in the map $\mathfrak{P}$ by the same constant. The rest of the proof is just looking at open neighbourhoods as in section \ref{Gapping}.
\subsection{Finishing the proof}
The theorem now follows from proposition \ref{Prop1}. Indeed, the zero-energy part of the map $\mathfrak{P}$ coincides with the map
$$p_{1_\ast}: MH^{G}_{\ast}(L_0\cap L_1)\to MH_{\ast}(\overline{L_0\cap L_1})$$
which by the usual Cartan isomorphism (theorem \ref{Cartan}) is an isomorphism. Since $\mathfrak{P}$ is gapped, this concludes the proof.

\section{Additional remarks}
\subsection{Invariance under Hamiltonian isotopy}
We show that equivariant Floer homology is, up to homotopy equivalence, invariant under Hamiltonian isotopy in the following sense:
\begin{theorem}
  Suppose $L_0,L_1$ intersect cleanly, and let $\varphi^{1}:X\to X$ be a $G$-equivariant Hamiltonian isotopy such that the intersection $L_0\cap \varphi^{1}(L_1)$ is clean, then there is a natural isomorphism
  $$FH^{G}_{\ast}(L_0,L_1;\Lambda)\xrightarrow{\cong} FH^{G}_{\ast}(L_0,\varphi^{1}(L_1);\Lambda).$$
\end{theorem}
\begin{obs}
  Here we work over the Novikov \emph{field} $\Lambda:=\Lambda_0[T^{-1}]$. This hypothesis is necessary.
\end{obs}
The construction of the chain maps is based on \emph{moving Lagrangian boundary conditions} which first appeared in \cite{Oh3}, we recall the basics. In this first approach we just assume $L_0,L_1$ are compact, relatively exact Lagrangians in a symplectic manifold $X$ and in clean intersection.

Consider a smooth, non-decreasing function $\beta:\R\to [0,1]$ such that
$$\begin{cases}
  \beta(s)=0 & s\leq 0,\\
  \beta(s)=1 & s\geq 1.
\end{cases}$$
Given a Hamiltonian isotopy $\varphi^{t}$ such that $L_0$ and $\varphi^1(L_1)$ intersect cleanly and a time dependent $J_{s,t}$ that for $s<<0$ coincides with a fixed $J_t$ and as $s>>0$ coincides with $J_t'$ for $J_t$ and $J_t'$ fixed, consider the moduli space of maps $u:\R\times[0,1]\to X$ such that
$$\begin{cases}
  \overline{\partial}_{J_{s,t}}u=0,\\
  u(s,0)\subset L_0, u(s,1)\in \varphi^{\beta(s)}(L_1),\\
  E(u)<\infty.
\end{cases}$$
The same properties as with fixed Lagrangian boundary conditions hold. In particular, we have exponential convergence to intersection points as $s\to \pm \infty$. We denote by $\mathcal{MOV}(J_{s,t})$ the moduli space of such $u$'s. For generic $J_{s,t}$ this space is transversely cut out and the evaluation maps are submersive. One then constructs the moduli space of cascades where one of the jumps has a moving boundary conditions. As pointed out in \cite{Biran-Cornea2} section 3.2, there is a subtlety in the weight we associate to a jump with moving boundary conditions: weighing it by its symplectic area does not give a chain map. Instead to a strip with moving conditions $u$ we associate the quantity $\omega(\tilde{u})$ where
$$\tilde{u}(s,t)=\left(\varphi^{t\beta(s)}\right)^{-1}\circ u(s,t).$$
In general, this quantity can be negative and that is why we need to invert $T$ in our coefficients. We then consider the usual count of moving cascades and weigh them with this in consideration to obtain a map

$$\Phi:FH_{\ast}(L_0,L_1;\Lambda)\to FH_{\ast}(L_0,\varphi^{1}(L_1);\Lambda).$$

To see that it is an isomorphism we do a gluing argument, namely we consider the function $\beta'(s)=\beta(-s)$ and get an interpolation of boundary conditions going from one to the other. When the two interpolating discs collide we can glue them and homotope the boundary condition to be constant, but since we are not quotienting by the $\R$-action on this strip the cascade has to be constant and thus is the identity. Counting the boundary components of this glued 1-dimensional space gives the chain homotopy between the identity and the composition.

Now we turn to the equivariant setting. Assume there is a $G$-equivariant Hamiltonian isotopy $\varphi^{t}$. By averaging the generating function, we can assume that it is obtained by a $G$-invariant smooth map $H_t:X\to \R$. The lift
\begin{eqnarray*}
  \tilde{H}_t: X\times T^{\ast}EG_n&\to & \R\\
  (x,(q,p))&\mapsto &H_t(x)
\end{eqnarray*}
is invariant with respect to the diagonal action and so we get a Hamiltonian isotopy $\tilde{\varphi}^{t}:X_n\to X_n$. To check it defines a map from $CF_{\ast}^{G}(L_0,L_1;\Lambda)\to CF_{\ast}^{G}(L_0,\tilde{\varphi}^{1}(L_1))$ we need to verify the maps constructed via moving Lagrangian boundary conditions commute with the increment maps up to homotopy. In order to do so we introduced a moduli space that interpolates the quilted jumps from one moving boundary condition to the other.

Pick $\hat{\beta}(s,r):\R\times \R\to \R$ with the following properties:
\begin{itemize}
\item Fixing any $r_0$, $\hat{\beta}(s ,r_0)$ is non-decreasing and is equal to $1$ for $s>>0$ and $0$ for $s<<0$.
\item In a neighbourhood of $s=0$, $\hat{\beta}(s,r)$ is constant in $s$ for all $r$.
\item There exists $r_0<<0$ so that for all $r\leq r_0$, $\hat{\beta}(s,r)=\beta(s+r)$. Similarly there exists $r_1>>0$ so that for all $r\geq r_1$ we have $\hat{\beta}(s,r)=\beta(s-r)$.
\end{itemize}
Now for $r\in \R$, we set the boundary conditions for the quilted strip to be $\tilde{\varphi}^{\hat{\beta}(s,r)}(L_1^{n})$ on the top boundary. Note that since each $\hat{\beta}(s,r_0)$ is locally constant around $0$ in $s$ the strip-like ends vary smoothly and the dimension of the intersection of the increment correspondence is independent of $r$.

For a generic choice of family of quilted data the moduli space of pairs $(u,r)$ where $u$ has the above boundary conditions and $r\in \R$ is transversely cut-out, we call such $u$ a moving quilt. The usual compactness results apply, and the boundary consists of:
\begin{itemize}
\item As $r\to -\infty$ there is strip breaking into a strip with moving boundary conditions followed by an increment quilt.
\item Similarly as $r\to \infty$ there is strip breaking into an increment quilt followed by a strip with moving boundary conditions.
\item Intermediate breaking into a moving quilt and regular Floer strips at either end.
\end{itemize}
We now consider the $1$-dimensional component of the space of cascades where one of the strips is a strip with moving boundary conditions and another one is an increment strip or one of the strips is a moving quilt for some $r$. Loss of compactness of this moduli space could a priori happen:
\begin{itemize}
\item When a moving boundary condition disc touches a quilted increment, in which case we glue them to a moving quilt.
\item Similarly if a moving quilt breaks then it is glued to the previous case.
\item Cascade breaking between the moving strip and the increment quilt. Depending on the relative positions this contributes to either $\Phi\circ \alpha_n$ or $\alpha_n\circ \Phi$ where $\Phi$ is obtained by the usual moving boundary cascades.
\item Cascade breaking into a differential cascade and a zero-dimensional piece of the moduli space. Depending on the relative positions this contributes to $H\circ \partial$ or $\partial \circ H$ where $H$ is obtained by the count of the zero-dimensional components.
\end{itemize}
In short, combining the 3rd and 4th cases above we see that the boundary is algebraically encoded in
$$\Phi\circ \alpha_n+\alpha_n\circ \Phi= H\circ \partial+\partial\circ H.$$
Hence the maps commute up to homotopy and we get a well defined map
$$\tilde{\Phi}:CF_{\ast}^{G}(L_0,L_1;\Lambda)\to CF_{\ast}^{G}(L_0,\varphi^{1}(L_1);\Lambda).$$
Similarly as in the non-equivariant case we construct an inverse map by interpolating the boundary conditions backwards. It is routine to check this is chain homotopic to the identity.
\subsection{The monotone setting}
In a symplectic manifold $X$ there are two homomorphisms
$$\langle\omega,-\rangle: \pi_2(X)\to \R, \hspace{1cm}c_1:\pi_2(X)\to \Z$$
where $c_1$ denotes the first Chern class of $TM$. We say that $M$ is $\tau$-\emph{monotone} if there exists a positive constant $\tau>0$ such that
$$\langle \omega,-\rangle= 2\tau c_1.$$
Similarly, given a Lagrangian $L\subset X$, the symplectic area and the Maslov index give homomorphisms from $\pi_2(X,L)$ to $\R$ and $\Z$ respectively, which we denote by $\langle\omega,-\rangle$ and $\operatorname{Mas}$. We say that $L$ is \emph{$\tau$-monotone} if there exists positive $\tau$ such that $\langle \omega,-\rangle=\tau\operatorname{Mas}$.
Note that if $L\subset X$ is $\tau$-monotone and $X$ is connected then $X$ is automatically $\tau$-monotone. We define the minimal Chern number of $M$ as
$$c_{M}:=\min\{c_1(A)\mid A\in \pi_2(X),c_1(A)>0\}$$
and the minimal Maslov number of $L$ as
$$N_{L}:=\min\{\operatorname{Mas}(D)\mid D\in \pi_2(X,L), \operatorname{Mas}(D)>0\}.$$
We can replace the relatively exact condition \eqref{RE} by the following
\begin{equation}\tag{M1}\label{M1}
	L_0\text{ and }L_1\text{ are monotone with }N_{L_j}\geq 3.
\end{equation}
This is done by replacing all statements of the form ``we get (or have) a disc in $\pi_2(X,L_j)$ and so it must have zero area and be constant" to ``we get (or have) a disc in $\pi_2(X,L_j)$ and since $N_{L_j}\geq 3$, its index is too high and thus the principal component of the moduli space lives in a space of negative dimension, which is empty". The following lemma is also used \cite[Proposition 4.22]{Cazassus}:
\begin{lema}
	If $L\subset X$ is monotone with minimal Maslov number $N_L$ and $G$ acts freely on $\mu^{-1}(0)$, then $L\slash G$ is also monotone, and furthermore $N_{L\slash G}$ is a multiple of $N_L$.
\end{lema}
This shows that all the Lagrangians $L_j^{n}$ are monotone with $N_{L_j}\geq 3$.
With these changes in mind, the construction and Theorem \ref{MainTheorem} go through over $\Lambda_0$.

With additional hypotheses, one can reduce the coefficient ring to a polynomial ring. Suppose the pair $L_0$ and $L_1$ satisfies the following ``annuli-monotonicity''
\begin{equation}
	\label{M2}\tag{M2}
	\begin{split}
		\text{There exists }\tau >0\text{ such that }
		\text{for all } u:[0,1]^2\to X\\\text{ with } u(s,j)\subset L_j\text{ and }u(0,t)=u(1,t)\text{ we have }\\
		\int_{[0,1]^2}\omega^{\ast}u=\tau \operatorname{Mas}(u).\hspace{2cm}
	\end{split}
\end{equation}
Combining \eqref{M1} and \eqref{M2} yields an energy-index relation as in \cite[(4) in remark 2.2]{Wehrheim-Woodward1}: given two critial points $p,q$ there exists a constant $c(p,q)$ such that for all cascades $u$ from $p$ to $q$ we have
$$\omega(u)=\operatorname{Mas}(u)+c(p,q).$$
This bounds the area of all cascades with $p$ and $q$ fixed, and so, by Gromov compactness, there are only finitely many homotopy classes (and thus powers of $T$) contributing to the differential. Moreover, the only powers of $T$ which appear are integer powers of $t=T^\tau$. This reduces the coefficient ring to be the polynomial ring $\Z_2[t]$ instead of $\Lambda_0$. By using the fibrations of the symplectic Borel spaces, one gets similar energy-index relations for the quilts involved in increment maps, which ensures that $\alpha_n$ takes values in the polynomial ring. If $L_0,L_1$ are $G$-invariant and satisfy \eqref{M2} on top of \eqref{M1} then so do all pairs $(L_0^{n},L_1^{n})$. One then constructs:
$$CF^{G}_{\ast}(L_0,L_1;\Z_2[t]):= \overline{\operatorname{Tel}}(FC_{\ast}(L_0^{n},L_1^{n};\Z_2[t])).$$

However the reductions $(\overline{L_0},\overline{L_1})$ may not satisfy \eqref{M2}, they do for example if $G$ is connected. However, we still have an energy-index relation on $X\slash\slash G$ and so the Floer complex is still defined over $\Z_2[t]$. Theorem \ref{MainTheorem} holds also in this setting. One might worry that the $t$-adic filtration is not complete, however, the induced filtration in homology is complete. Then \cite[Theorem 3.9]{McCleary} can be applied to prove Theorem \ref{MainTheorem}.

As a final note, setting value $t=1$ yields \cite{Cazassus}'s version of equivariant Floer theory, and Theorem \ref{MainTheorem} holds in this setting.

\appendix
\section{Removal of the technical hypothesis}\label{Appendix}
We discuss here how to remove the technical hypothesis at the cost of using so-called ``domain dependent'' almost complex structures. The problem arises when we try to achieve transversality for tuples of pseudoholomorphic discs that are \textbf{not} absolutely distinct. To resolve this issue we proceed as follows. A cascade has an underlying metric binary tree where each vertex represents a Floer strip. We can choose an almost complex structure for each vertex independently, with the tuple of almost complex structures depending on the underlying tree. This, of course, is subject to some conditions: for one when two discs (i.e. vertices) collide, we need to ensure that the two discs have the same almost complex structure attached so we can glue them together, this is a recursive definition of the boundary of the moduli space, where the length of an edge being zero can be identified with a tree where one of the colliding vertices is removed. Second, then the length of an edge tends to infinity the tree `breaks' into two smaller trees, we must require the almost complex structures to agree with those of the already prescribed smaller trees. We spell out the details.

We define $\mathcal{T}_n$ to be the set of equivalence classes of trees with $n+2$ vertices $(T,E)$ equipped with a `length' function $\ell:E\to [0,\infty]$ subject to the following conditions:
\begin{itemize}
\item All vertices have valence at most 2, and there are exactly two vertices of valence 1, these are called \emph{exterior vertices}. One of them is labelled as starting point and the other one as endpoint.
\item The edges adjacent to exterior vertices (called exterior edges) have infinite length and all other edges (interior edges) have finite length. Note that these edges could be the same in the case $n=0$.
\end{itemize}
Where two trees are said to be \emph{equivalent} if there exists a graph isomorphism that preserves the length functions and identifies the starting point and the endpoint.

The moduli space $\mathcal{T}_n$ admits the structure of a smooth manifold with corners and is diffeomorphic to $[0,\infty)^{n-1}$ for $n\geq 2$, and diffeomorphic to a point if $n=0,1$. The corners correspond to the number of interior edges with length zero. The number of zeros will be called the \emph{depth} of the corner.

There is an identification between the depth $k$-corner of $\mathcal{T}_n$ and $\mathcal{T}_{n-k}$ where to a representative with edges $e_1,\dots, e_k$ and $\ell(e_1)=\dots=\ell(e_k)=0$ we associate the metric tree with these edges removed.

We compactify $\mathcal{T}_n$ in the following way: when the length of an edge tends to infinity, the tree breaks into a pair of trees: we add a vertex where the edge breaks, which is the endpoint of one and the starting point of the other. In other words these boundary components are products $\mathcal{T}_{k_1}\times \mathcal{T}_{k_2}$ with $k_1+k_2=n$. We denote by $\overline{\mathcal{T}_n}$ this compactification.

As an example, the space $\mathcal{T}_2$ consists of an open interval $[0,\infty)$, while $\overline{\mathcal{T}_2}$ is precisely the closed interval $[0,\infty]$. Similarly $\mathcal{T}_3$ is the open corner $[0,\infty)^2$ and $\overline{\mathcal{T}_3}$ is diffeomorphic to the square $[0,\infty]^2$.

There is a \emph{universal tree} living living over these spaces, $\mathscr{T}_n$.
\begin{figure}[h]
  \centering
  \begin{tikzpicture}
    \begin{scope}
      \draw (0,-2)node[draw,fill=white,circle,inner sep=1pt]{} --node[fill,circle,inner sep=1pt]{}node[right]{$\ell=0$} (0,2)node[draw,fill=white,circle,inner sep=1pt]{};
    \end{scope}
    \begin{scope}[shift={(2,0)}]
      \draw (0,-2)node[draw,fill=white,circle,inner sep=1pt]{} -- (0,2)node[draw,fill=white,circle,inner sep=1pt]{};
      \path (0,-1)node[fill,circle,inner sep=1pt]{}--(0,1)node[fill,circle,inner sep=1pt]{};
      \draw[<->] (0.3,-1)--node[right]{$\ell$}+(0,2);
    \end{scope}
    \begin{scope}[shift={(4,0)}]
      \draw (0,-2)node[draw,fill=white,circle,inner sep=1pt]{} -- (0,-0.2)node[draw,fill=white,circle,inner sep=1pt]{};
      \draw (0,2)node[draw,fill=white,circle,inner sep=1pt]{} -- (0,0.2)node[draw,fill=white,circle,inner sep=1pt]{};
      \path (0,-1.1)node[fill,circle,inner sep=1pt]{}--(0,1.1)node[fill,circle,inner sep=1pt]{};
      \draw[<->] (0.3,-1.1)--node[right]{$\ell=\infty$}(0.3,-0.2);
      \draw[<->] (0.3,1.1)--node[right]{$\ell=\infty$}(0.3,0.2);
    \end{scope}
    \draw (0,-2.5)node[fill,circle,inner sep=1pt]{}node[below]{$\ell=0$}--node[below]{$\overline{\mathcal{T}_{2}}$}(4,-2.5)node[fill,circle,inner sep=1pt]{}node[below]{$\ell=\infty$};
  \end{tikzpicture}
  \caption{The compactification of $\mathcal{T}_{2}$}
  \label{fig:compact-t2}
\end{figure}
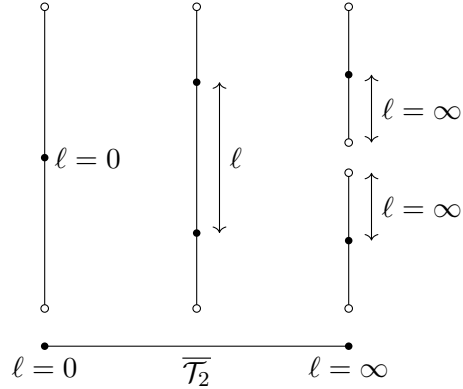

Now assume we are in the setting of clean intersection Floer homology. To the universal tree above $\overline{\mathcal{T}_n}$, we associate to each vertex a path of admissible almost complex structures $J_t$ that vary smoothly as a family over $\overline{\mathcal{T}_n}$. Formally a smooth map $\mathscr{J}_t:\mathcal{T}_n\to (\mathcal{C}^{\infty}([0,1],\mathcal{J}_{\operatorname{adm}}))^{n}$.

We require they respect the following recursive conditions:
\begin{itemize}
\item In a corner of length zero, the almost complex structures agree with the almost complex structures of $\mathcal{T}_{n-k}$ via the mentioned identification, in particular the two colliding vertices have the same associated almost complex structure.
\item When the tree breaks, we require that the associated almost complex structures of the two resulting trees agree with the almost complex structures already defined for each of the trees.
\end{itemize}
We define the moduli space $\widetilde{\mathcal{M}}_n$ to be the set of tuples $(u_1,\dots, u_n,T)$ where the tuple is $\mathscr{J}_t(T)$-holomorphic. We say that such a tuple is \emph{transversely cut-out} if the Cauchy-Riemann operators are surjective, the evaluation maps of each $u_j$ are submersive and so is the map $(u_1,\dots, u_n,T)\mapsto T$.

We then define $\widetilde{\mathcal{M}}_n(p,q)$ to be the fibered product
$$\begin{tikzcd}
  &  \widetilde{\mathcal{M}}_m \arrow{d}{\operatorname{EV}}\\
  W^{u}(p)\times(L_0\cap L_1)^{n-1}\times W^{s}(q)\times \mathcal{T}_n\arrow{r}{\Psi} & (L_0\cap L_1)^{2n} \times \mathcal{T}_n.
\end{tikzcd}$$
Where
$$EV:(u_1,\dots, u_n,T)\mapsto (u_1(-\infty),u_1(+\infty),\dots, u_n(+\infty),T)$$
and
$$\Psi:(p,q_1,\dots,q_n,r, T)\mapsto (p,q_1,\varphi^{\ell(e_1)}(q_1),\dots, r, T)$$
where $\varphi$ is the gradient flow of the Morse function and $e_1,\dots, e_{n-1}$ are the edges of $T$ ordered from starting point to endpoint.

The virtual dimension of $\widetilde{\mathcal{M}}(p,q)$ at a tuple $(u_1,\dots, u_n,T)$ is
$$\ind_{f}(p)-\ind_{f}(q)+\operatorname{Mas}(u_1\#\dots\#u_n)+n-1.$$
Quotienting out by the free $\R^{n}$ action we get the moduli space $\mathcal{M}_{n}(p,q)$.

We now show that transversality is achieved when the virtual dimension of $\mathcal{M}_{n}(p,q)$ is less or equal to 1 inductively. When $n=1$ it is clear that for generic $\mathscr{J}_t$ this moduli space is transversely cut-out, since there is only one strip. Now assume transversality is achieved in $\mathcal{T}_{n-1}$, using a universal moduli space it is straightforward to see that we can extend it to the interior, except for one problem. When two discs touch, i.e. $\ell(e)=0$ for some interior edge, our construction of $\mathscr{J}_t$ forces the two discs to have the same almost complex structure, and it would be a problem if they were not distinct. We rule out this from happening:\\
If the two discs are the same $u_{j}=u_{j+1}=u$, then $u(-\infty)=u(+\infty)$. However by transversality of $\mathcal{T}_{n-1}$, this disc is transversely cut out, and the moduli space of discs that start and end at the same point is empty unless the Maslov index of $u$ is greater or equal to $1$. But then, removing the two discs we would get a tree of virtual dimension $-1$: this space is transversely cut-out and thus empty, a contradiction.

Compactness of the moduli space is essentially the same: if two pseudoholomorphic discs touch this occurs as a boundary point of $\overline{\mathcal{T}_n}$, by gluing we can smoothly identify it with an interior point of $\mathcal{T}_{n-1}$, thus the only lack of compactness comes from the boundary breaking when the length of an edge goes to infinity and this is precisely the cascade breaking that shows that the differential squares to zero.

The rest of the proofs carry over using this formalism, for example the space of quilt-cascades just requires the addition of a distinguished vertex which corresponds to the quilted-disc, and the homotopies involved have a vertex which a label $R$ which corresponds to the distance between the two seams.

Note that the gapping condition remains unchanged since we can achieve transversality for almost complex structures in that domain.

\bibliographystyle{amsalpha}
\bibliography{references}
\end{document}